\documentclass{daj}

\dajAUTHORdetails{%
  title = {Universality of the minimum modulus for\\ random trigonometric polynomials}, 
  author = {Nicholas A.\ Cook and Hoi H. \ Nguyen},
  plaintextauthor = {Nicholas A.\ Cook and Hoi H. \ Nguyen},
  plaintexttitle = {Universality of the minimum modulus for random trigonometric polynomials}, 
  runningtitle = {The minimum modulus for random trigonometric polynomials}, 
  runningauthor = {Nicholas A.\ Cook and Hoi H. \ Nguyen},
  copyrightauthor = {N.\ A.\ Cook and H.\ H. \ Nguyen},
  keywords = {Random trigonometric polynomials, Universality, Edgeworth expansion},
}   

\dajEDITORdetails{%
   year={2021},
   number={20},
   received={5 February 2021},   
   published={6 October 2021},  
   doi={10.19086/da.28985},       
}   

\usepackage{mathrsfs}
\usepackage{amsmath}
\usepackage{amsthm}
\usepackage{amsfonts}
\usepackage{amssymb}
\usepackage{url}
\usepackage{enumerate}
\usepackage{bbm}
\usepackage{cleveref}

\theoremstyle{plain}
  \newtheorem{theorem}{Theorem}[section]

  \newtheorem{prop}[theorem]{Proposition}
  \newtheorem{proposition}[theorem]{Proposition}
  
  \newtheorem{lemma}[theorem]{Lemma}
  \newtheorem{corollary}[theorem]{Corollary}

  \newtheorem{claim}[theorem]{Claim}
  
\theoremstyle{definition}
  \newtheorem{definition}[theorem]{Definition}
  \newtheorem{remark}[theorem]{Remark}

\numberwithin{equation}{section}
\numberwithin{theorem}{section}

\renewcommand\Re{{\operatorname{Re}}}
\renewcommand\Im{{\operatorname{Im}}}

\newcommand\R{{\mathbb{R}}}
\newcommand\C{{\mathbb{C}}}

\renewcommand\P{{\mathbf{P}}}
\newcommand\E{{\mathbf{E}}}

\newcommand\bw{{\mathbf{w}}}

\newcommand\Z{{\mathbb{Z}}}

\newcommand\al{\alpha}
\newcommand\la{\lambda}
\newcommand\1{\mathbf{1}}

\newcommand\Ba{{\mathbf a}}
\newcommand\Bb{{\mathbf b}}

\newcommand\Bw{{\mathbf w}}
\newcommand\Bx{{\mathbf x}}

\newcommand\BN{{\mathbf N}}

\renewcommand\Pr{{\mathbf P }}

\newcommand\CA{{\mathcal A}}

\newcommand\cD{{\mathcal D}}
\newcommand\CE{{\mathcal E}}

\newcommand\CG{{\mathcal G}}

\newcommand\CM{{\mathcal M}}
\newcommand\CN{{\mathcal N}}

\newcommand\CU{{\mathcal U}}

\newcommand\CX{{\mathcal X}}

\renewcommand \small {\scriptsize}

\newcommand\eps{\varepsilon}

\newcommand\wh{\widehat}

\renewcommand\a{\alpha}

\newcommand\bs{\backslash}

\newcommand{\wb}{\overline}
\newcommand\Bg{{\CN_{\R}(0,1)}}
\renewcommand\1{\mathbbm 1}
\newcommand\til{\widetilde}

\newcommand{\etalchar}[1]{$^{#1}$}

\newcommand\nc\newcommand
\nc\dmo\DeclareMathOperator
\renewcommand\bs\boldsymbol
\definecolor{purple}{rgb}{0.9,0,0.8}
\nc{\nick}[1]{{\color{purple} #1}}
\nc{\red}[1]{{\color{red} #1}}
\nc\lf{\lfloor}
\nc{\rf}{\rfloor}
\nc\pr{{\mathbf{P}}}
\nc\lam{{\lambda}}
\nc\tP{{\til P}}
\nc{\xxi}{\xi}
\nc{\ii}{{\sqrt{-1}}}
\dmo{\Leb}{{Leb}}
\nc{\wt}{\widetilde}
\nc{\mLeb}{m_{\Leb}}
\nc{\tran}{{\mathsf{T}}}
\nc{\mom}{{m}}
\nc{\avec}{{\bs a}}
\nc{\bvec}{{\bs b}}
\nc{\uvec}{{\bs u}}
\nc{\vvec}{{\bs v}}
\nc{\wvec}{{\bs w}}
\nc{\ul}{\underline}
\nc{\us}{{\bs{s}}}
\nc{\bbs}{{\bs{s}}}
\nc{\pol}{{P}}
\nc{\pt}{Q}
\nc{\pk}{P}
\nc{\ptt}{{\wt{P}}}
\nc{\pkt}{{\wt{Q}}}
\nc{\tor}{{\R/\Z}}
\nc{\mk}{{m}}
\nc{\mt}{{\wt m}}
\dmo\bad{{bad}}
\nc{\badarcs}{{E_{\bad}}}
\nc{\Cp}{{K_0}}
\nc{\Smooth}{{K}}
\nc{\smooth}{{\kappa}}
\nc{\Spread}{{L}}
\nc{\spread}{{\lambda}}
\nc{\dbd}{{K}}
\nc{\tpol}{{\til\pol}}
\nc{\ali}{{\al_i}}
\nc{\alo}{{\al_1}}
\nc{\alm}{{\al_\mom}}
\nc{\dil}{{L}}
\nc{\mat}{{A}}
\dmo{\row}{{row}}
\nc{\bzeta}{{\bs \zeta}}
\nc{\bxi}{{\bs \xi}}
\nc\height{{\tau}}
\nc\subG{{K_{G}}}
\nc\Points{{\CM}}
\nc\Char{{\Phi}}
\nc\charr{{\phi}}
\nc{\biggish}{{\nick{\log^2n}}}
\nc{\bigger}{{\nick{\log^3n}}}

\nc{\NIQ}[1]{{\color{red} \sf $\spadesuit\spadesuit\spadesuit$ Nick: [#1]}}
\nc{\NIC}[1]{{\color{blue} \sf $\spadesuit\spadesuit\spadesuit$ Nick: [#1]}}
\nc{\revised}[1]{{\color{black} #1}}

\begin{document}

\begin{frontmatter}[classification=text]

\title{Universality of the minimum modulus for\\ random trigonometric polynomials} 

\author[nc]{Nicholas A.\ Cook}
\author[hn]{Hoi H. \ Nguyen \thanks{The second author is supported by NSF CAREER grant DMS-1752345.}}

\begin{abstract}
It has been shown in \cite{YaZe} that the minimum modulus of random trigonometric polynomials with Gaussian coefficients has a limiting exponential distribution. We show this is a universal phenomenon. Our approach relates the joint distribution of small values of the polynomial at a fixed number $m$ of points on the circle to the distribution of a certain random walk in a $4m$-dimensional phase space. Under Diophantine approximation conditions on the angles, we obtain strong small ball estimates and a local central limit theorem for the distribution of the walk.
\end{abstract}
\end{frontmatter}





\section{Introduction}

Consider 
the Kac polynomial 
\begin{equation}	\label{Kac}
F_n(z) = \sum_{j=0}^n \xi_jz^j
\end{equation}
for a sequence of iid random variables $\xi_j$ (real or complex). 
The study of the distribution of zeros of $F_n$, and in particular on the number of real zeros, has a long history: the case that $\xi_j\in\{-1,0,1\}$ was considered by Bloch and Polya \cite{BlPo} and Littlewood and Offord \cite{LiOf38,LiOf43} in the 1930s, and the Gaussian case by Kac in the 1940s \cite{Kac1,Kac2}. We refer to \cite{TV} for an overview of the vast literature inspired by those early works.


To the best of our knowledge, the question of the size of the minimum modulus over the unit circle for Kac polynomials
was first raised by Littlewood \cite{Lo}, who considered the case of Rademacher signs $\xi_j = \pm1$.\footnote{We also refer the readers to \cite{BM} for a recent striking result answering another question of Littlewood.}
In particular, Littlewood asked whether $\min_{|z|=1}|F_n(z)|=o(1).$\footnote{Here and throughout the article asymptotic notation is with respect to the limit $n\to \infty$; see \Cref{sec:notation} for our notational conventions.}
This question was answered in the affirmative by Kashin \cite{Ka}; a significant improvement was later obtained by Konyagin \cite{K}, who showed
\begin{equation}	\label{Konyagin}
\P\Big(\, \min_{|z|=1}|F_n(z)| \ge n^{-1/2+\eps} \,\Big) \to 0
\end{equation}
as $n\to \infty$, for any $\eps>0$. 
Subsequently, Konyagin and Schlag \cite{KSch} showed that for any $\eps>0$,
\revised{
\begin{equation}	\label{KS:UB}
\limsup_{n\to\infty}
\P\Big(\, \min_{|z|=1}|F_n(z)| \le \eps n^{-1/2}\,\Big) \le C\eps
\end{equation}
for 
a universal constant $C<\infty$.} 
From the above two estimates, it is thus natural to ask whether $n^{1/2}m(F_n)$ converges in law, and to identify the limiting distribution. 

This question was recently addressed for the case of Gaussian coefficients by a beautiful result of Yakir and Zeitouni \cite{YaZe}, which we now recall.
As we consider the restriction of $F_n$ over the unit circle we parametrize $z=e(x)$, where here and throughout we abbreviate $e(t):= \exp( \ii t)$.
The work \cite{YaZe} considers the normalized trigonometric series
\begin{equation}	\label{model:YZ}
P_n(x) = \frac{1}{\sqrt{2n+1}}\sum_{j=-n}^n \xi_j e(jx), \qquad x\in \R,
\end{equation}
where $\xi_j$ are iid copies of a real or complex, centered random variable $\xi$ of unit variance.
Note that $P_n$ has been scaled to have 
\revised{unit variance at each fixed $x$. Up to a factor of unit modulus, which does not affect our results, $P_n$ is the restriction of the Kac polynomial $(2n+1)^{-1/2}F_{2n}(z)$ to the unit circle (all of our arguments extend to the case of odd degree).}
We denote
\begin{equation}
\mk_n:= \min_{x\in [-\pi,\pi]} 
 | \pk_n(x)|.
\end{equation}
With our normalization and from \eqref{Konyagin} and \eqref{KS:UB} we expect that $\mk_n$ is typically of order $n^{-1}$. 
For the case of Gaussian coefficients, in \cite{YaZe} the limiting distribution of $n\cdot m_n$ was shown to be exponential:

\begin{theorem}[{\cite{YaZe}}]	\label{thm:YZ}
Assume that $\xi$ is a standard real or complex Gaussian.
Then for any $\tau>0$, 
\begin{equation}	\label{YZ:lim}
\lim_{n\to \infty} \pr\Big(\, \mk_n>\frac{\tau}{n}\,\Big)= 
e^{-\lam \tau}
\end{equation}
where $\lam= 2\sqrt{\pi/3}$. 
\end{theorem}

As shown in \cite[Section 5]{YaZe}, their argument in fact extends to allow some distributions with a small Gaussian component -- specifically,  $\xi$ of the form 
\begin{equation}	\label{eqn:YZ:xi}
\xi'+\delta X
\end{equation}
with $\delta$ at least of order $n^{-1}\log n$, where $\xi'$ and $X$ are independent, $X\sim \CN_\R(0,1)$, and $\xi'$ is an arbitrary random variable satisfying Cram\'er's condition.
While Cram\'er's condition is weaker than assuming a bounded density, it does not allow $\xi'$ to be discrete.


In the present work we show that the limiting exponential law for $\mk_n$ is universal.
Here and in the sequel, $\P_{\CN_{\R}(0,1)}$ denotes a probability measure under which the real variables $\xi$ or $\xi',\xi''$ are standard Gaussian. 

\begin{theorem}[Main result]
\label{thm:main}
Assume $\xi$ is a centered sub-Gaussian variable of unit variance, which is either real-valued, or takes the form $\frac1{\sqrt{2}}(\xi'+\ii\xi'')$ for iid real variables $\xi',\xi''$.
Then for any $\tau>0$,
\begin{equation}
\P\Big(\, m_n> \frac{\tau}n\,\Big) - \P_{\CN_\R(0,1)} \Big(\, m_n >\frac{\tau}n \,\Big) \longrightarrow 0
\end{equation}
as $n\to\infty$.
\end{theorem}


\begin{remark}\label{rmk:extensions}
In the proof we treat the case \eqref{model:YZ} with real-valued $\xi$ -- the complex case is slightly simpler.
The necessary modifications\revised{,} 
as well as an extension to another model of random trigonometric series, are given in \Cref{sec:extensions}.
\end{remark}

\begin{remark}
The sub-Gaussianity assumption is mainly for convenience, and
one can check that for our arguments it suffices to assume $\xi$ has a finite moment of sufficiently large order.
\end{remark}

As an immediate consequence we extend \Cref{thm:YZ} to general sub-Gaussian coefficients:

\begin{corollary}	\label{cor:main}
The limit \eqref{YZ:lim} holds when $\xi$ is any sub-Gaussian random variable of mean zero and unit variance. 
\end{corollary}

In particular, \eqref{YZ:lim} holds for Rademacher polynomials, which were the focus of the aforementioned works of Littlewood and others. In fact, the Rademacher case in some sense captures the main challenges for our proof.
We comment on some of these challenges below.
See Figure \ref{fig1} for a numerical illustration of the universality phenomenon.

\begin{figure}[ht!] \label{fig1}
	\includegraphics[width=80mm]{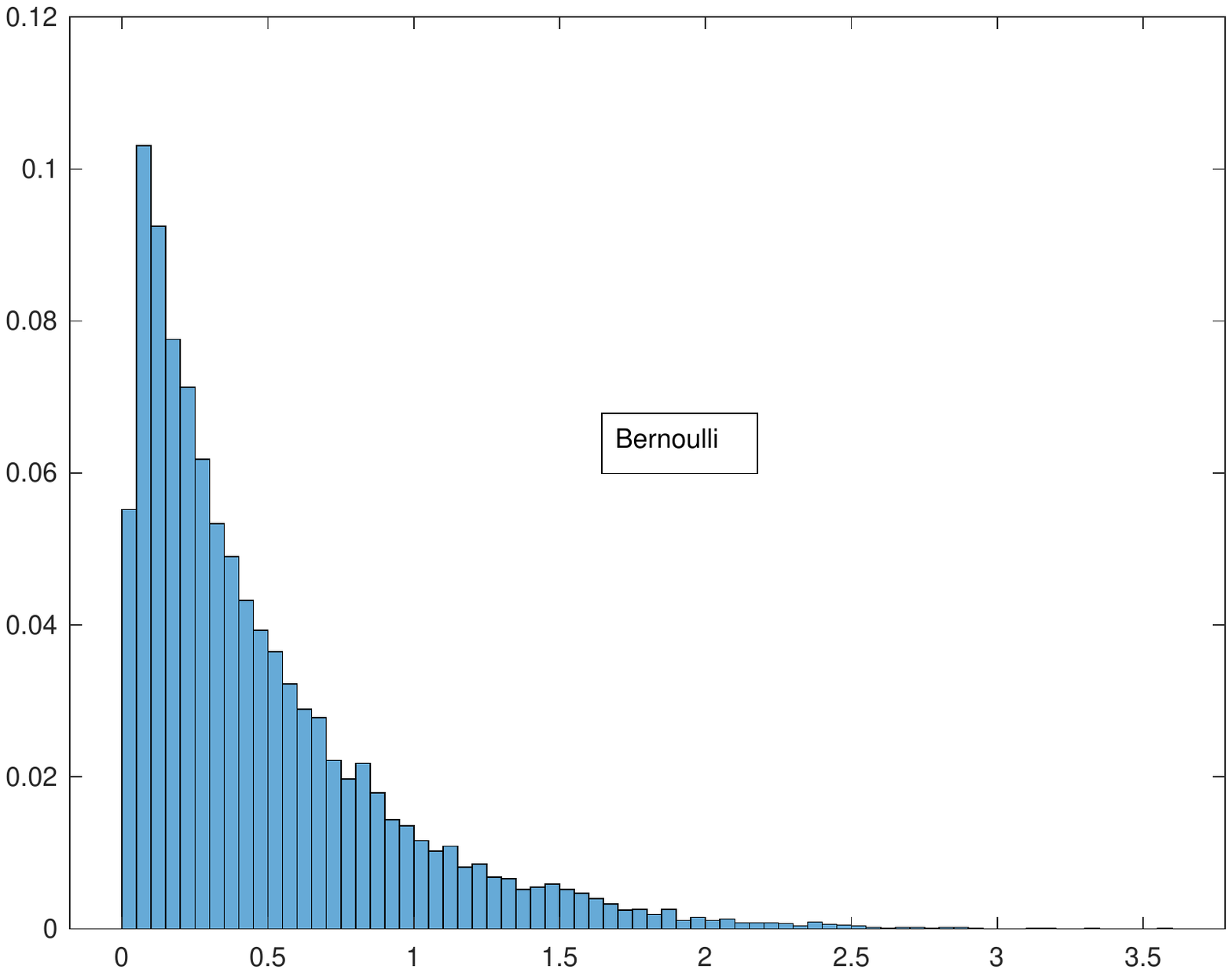}
	\includegraphics[width=80mm]{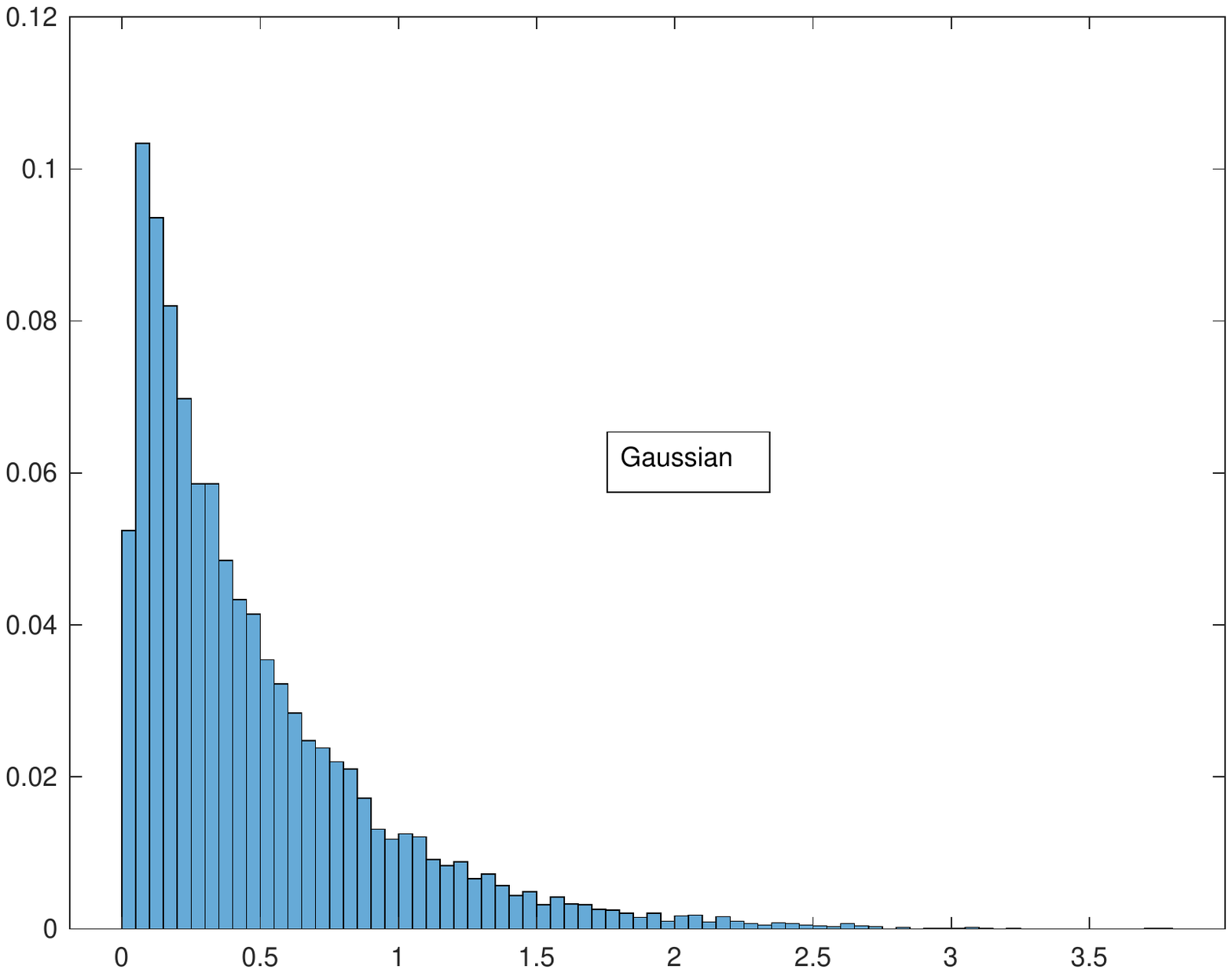}
	\caption{Histogram of the minimum modulus over $10^4$ points equally spaced points on the unit circle, for $10^4$ samples of a random degree 20 polynomial $P_n(x)$ of \eqref{model:YZ} with Rademacher (left) and Gaussian (right) coefficients.}
\end{figure}

We mention that the distribution of the \emph{maximum} value over a curve for various random analytic functions has been studied extensively; see for instances the books \cite{AT,AW} and the references therein.
Sharp asymptotics for the maximum of random trigonometric polynomials with Rademacher coefficients were obtained by Salem and Zygmund \cite{SaZy} and Hal\'asz \cite{Halasz73}, and extended to more general coefficient distributions by Kahane \cite{Kahane-book}. 
In recent years there has been particular focus on characteristic polynomials of random unitary matrices, with $\gamma$ the unit circle \cite{ABB,PaZe,CMN,CoZe}, and the Riemann zeta function on a randomly shifted unit interval on the critical axis \cite{ABBRS,Najnudel,Harper,ABR}. 
Such questions are closely tied to a fine understanding of large deviations and concentration of measure for values of the function at given points. 

The minimum modulus has received comparatively less attention. As we explain below, its behavior is governed by central limit theorems and anti-concentration for the distribution at given points. (Another well-known instance of the dichotomy of concentration/anti-concentration for large/small values of random fields is in the study of singular values of random matrices.)

We further note that proving universality for {\it roots} of classical random ensembles has become an active direction of research in recent years, see for instance \cite{BD2, DHNgV, DONgV, IKM, KZ, NgNgV, ONgV, TV} and the references therein.  Our main result stands out from the above works in two ways: that our focus is not on the statistics of roots, and our method is totally different. 
Corollary \ref{cor:main} can be seen as a polynomial analogue of the result \cite{TVleast} by Tao an Vu where they showed that the least singular value statistics of random iid matrices is universal, although there is no real connection between the random matrix model and our random polynomials. It is remarked that the study of both the minimum modulus of Kac polynomials and of the least singular values of random matrices have important implications to the study of the condition number of matrices, see for instance \cite{BG} and \cite{TVcond}.

Finally, we note that since the completion of this work, there has been progress on the related problem of the distance of the \emph{nearest root} of Kac polynomials to the unit circle. A beautiful result of Michelen and Sahasrabudhe \cite{MiSa} establishes the limiting distribution for the Gaussian case, resolving a conjecture of Shepp and Vanderbei \cite{ShVa}. 
In recent work with Yakir and Zeitouni \cite{CNYZ} we apply some tools developed in the present paper to show their result is universal. 

\subsection{Some comments on the proof}
\label{sec:highlevel}
We briefly sketch some highlights of the proof of \Cref{thm:main}.
Consider the parametrized random curve $\{P_n(x):x\in [-\pi,\pi]\}$ as the trajectory of a particle in the complex plane.
Following \cite{KSch} we approximate the time the particle is closest to the origin by a point in a discrete mesh $\CX=\{x_\al\}_{\al=1}^N\subset[-\pi,\pi]$. Since the velocity $P_n'(x)$ is typically of order $n$, in order to capture this moment we must take $N$ much larger than $n$. However, this means that each approach within distance $O(1/n)$ of the origin will carry several points $P_n(x)$, $x\in \CX$ near the origin, so that a union bound over events that $P_n(x_\al) = O(1/n)$ is too wasteful to isolate the distribution of $m_n$. Following \cite{YaZe}, we isolate a single time $x_\al\in \CX$ for each approach, so that $|P_n(x_\al)|$ is approximately a local minimum, by considering both $P_n(x_\al)$ and $P_n'(x_\al)$ -- the precise criterion is given in \Cref{sec:linearize}. 
The result is a collection of events $\CA_\al$, $\al\in [N]$, that $x_\al$ is an approximate local minimizer, with each event determined by the positions and velocities of the particle on the discrete set $\CX$. In this way we obtain a point process $\CM_n$ on $\R_+$ of approximate local minima $n|P_n(x_\al)|$, rescaled so that the global minimum is of order one.

For the Gaussian case, it was shown in \cite{YaZe} that $\CM_n$ is approximately a Poisson point process of intensity $2\sqrt{\pi/3}$, from which the result clearly follows. In \Cref{sec:YZ} we provide a sketch of their key argument using an invariance principle of Liggett.
For universality, our approach is to establish universality for the joint distribution of 
\[
S_n=S_n(\al_1,\dots, \al_m):=(P_n(x_{\al_i}), P_n'(x_{\al_i}))_{i\in [m]}\in \C^{2\mom}
\]
giving the positions and velocities of the particle at any fixed collection of times $x_{\al_1},\dots, x_{\al_m}$; this allows us to deduce universality for the global minimum by comparison of moments. 

The event that the real and imaginary parts of the positions and velocities
lie in given ranges, and moreover that $\CA_{\al_i}$ holds for each $i\in [m]$, is the event that the vector $S_n$ lies in a certain compact domain $\CU_n$ in $4m$-(real-)dimensional phase space. While $\CU_n$ has piecewise smooth boundary, its regularity depends strongly on $n$, so that estimating its measure under the law of $S_n$ requires precise estimates of the measure of boxes at polynomially-small scales. 

Recalling that $P_n$ is a trigonometric polynomial, we see that $S_n$ is a random walk of the form $\sum_{j=-n}^n\xi_j \wvec_j$, with $\wvec_j\in \R^{4m}$ giving the real and imaginary parts of $e(jx)$ and its derivative $je(jx)$ at the times $x_{\al_1},\dots, x_{\al_m}$. In particular, when the coefficients $\xi_j$ are Gaussian, $S_n$ is a Gaussian vector, and so the main problem is to obtain a quantitative central limit theorem for $S_n$ when the coefficients are general sub-Gaussian variables. This, as well as a small ball estimate, hinge on a strong decay estimate on the characteristic function of $S_n$ (\Cref{thm:char}), which is the main technical component of the proof.
(In fact our argument yields more than a CLT, giving a quantitative \emph{Edgeworth expansion} for the distribution of $S_n$, though for our purposes we only need that each term of the expansion is smooth.)

In our general setting and in particular when the coefficients have discrete distribution, the distribution of the polynomial and its derivative at given points $x_{\al_1},\dots, x_{\al_m}$ depends strongly on arithmetic properties of the $x_{\al_i}$ (compared to the complex Gaussian case of \Cref{thm:YZ} where the distribution is stationary under rotations.)
In particular, the desired control on the characteristic function does not hold for all choices of the $x_{\al_i}$ -- basically when two of the points are too close together or nearly antipodal, or when $e(x_{\al_i})$ is close to a root of unity of order $n^{o(1)}$ for some $i\in [m]$. 
We handle such ``bad'' $m$-tuples with relatively crude arguments (following \cite{KSch}), and establish the decay estimate on the characteristic function for ``nice'' tuples.

The latter is the most technically challenging part of the proof. 
A similar estimate for the case $\mom=2$ was obtained in \cite{DNN}, but the generalization to higher dimensions, together with the complexity of the case when $\xi$ is real-valued, pose significant challenges. 
For this, roughly speaking, we must show that it is not possible to simultaneously dilate the steps $\wvec_j$ of the walk by a factor $K$, for any $K=n^{O(1)}$, so that their projections $\psi_j$ in some common direction all approximately lie in the integer lattice.
We argue by contradiction, showing that if there is such a projection and dilation, then the sequence $\psi_j$ can be locally approximated by polynomial progressions of controlled degree. 
Here we crucially use the trigonometric properties of the steps $\wvec_j$. Combining this information with some judicious differencing manipulations, we can isolate an angle $x_i$ that is well-approximated by a rational of small denominator, contradicting the smoothness assumption.

To summarize, some highlights of our note include:
\begin{enumerate}
\item  
A nearly sharp characterization, in terms of arithmetic properties, of the collection of arcs of the circle over which the Kac polynomial is strongly approximated by a Gaussian Kac polynomial (in the sense of joint distributions at any fixed number of points);
\item Sharp small ball estimates under microscopic scaling for random walks in $\R^m$ of the form $\sum_j\xi_j (g(\frac{jt_1}{n}),\dots, g(\frac{jt_m}{n}))$ for various smooth functions $g:S^1\to \C$, such as $e(x)$, or $x\sin x$;
\item Local limit theorems for such high-dimensional random walks;
\item A sub-polynomial decay estimate on the associated characteristic function, which greatly improves on estimates from \cite{KSch}. 
\end{enumerate}
 All of these results seem to be new and of independent interest.

\subsection{Organization}

In \Cref{sec:prelim} we will discuss the proof of \cite{YaZe} and reduce our task to establishing \Cref{prop:joint}, establishing universality for the joint distribution of low-lying near-local minima over a discrete subset of the torus. 
Along the way we recall some lemmas from \cite{YaZe}, and identify two important arithmetic properties for collections of points in the torus that will be crucial for subsequent analysis.
\Cref{sec:RW} reformulates \Cref{prop:joint} in terms of a vector-valued random walk, and proves it using a small-ball estimate (\Cref{thm:smallball}) and local central limit theorem (\Cref{thm:box}), which are consequences of a strong decay estimate for the characteristic function (\Cref{thm:char}).
The deduction of the main result from \Cref{prop:joint} is given in Sections \ref{sec:proof-moments} and \ref{sec:proof-main}.
\Cref{thm:smallball} and \Cref{thm:box} are deduced from \Cref{thm:char} in Sections \ref{sec:smallball} and \ref{sec:box}, respectively, and \Cref{thm:char} is proved in \Cref{sec:char}.
Finally, in \Cref{sec:extensions} we describe how our result can be extended to other models of random trigonometric polynomials.

\subsection{Notation}
\label{sec:notation}

We write $C,C', C_0, c$ etc.\ to denote positive absolute constants, which may change from line to line, while $C(\tau)$ etc.\ denotes a constant that depends only on the parameter (or set of parameters) $\tau$.
We use the standard asymptotic notation $f=O(g)$, $f\ll g$ and $g\gg f$ to mean $|f|\le Cg$ for some absolute constant $C>0$, and $f=O_\tau(g)$, $f\ll_\tau g$ and $g\gg_\tau f$ to mean $|f|\le C(\tau) g$. For positive sequences $\{f_n\}, \{g_n\}$ we say that $g_n=o(f_n)$ and $f_n =\omega(g_n)$ if $\lim f_n/g_n \to \infty$ with $n$.
We allow implied constants to depend on the sub-Gaussian constant of $\xi$ without explicitly indicating this.

For a real number $x$, $\|x\|_{\R/\Z}$ denotes the distance from $x$ to the nearest integer, and $m=\mLeb(\cdot)$ denotes the Lebesgue measure on $\R^d$ for any $d$. For a compact interval $J\subset \R$ we write $|J|:=\mLeb(J)$ for its length. 
$\{t\}= t-\lf t\rf$ denotes the fractional part of $t\in \R$.
We write $e_n(\theta)$ for $e(\theta/n)$.
The singular values of a matrix $M$ are ordered
$\sigma_1(M)\ge \sigma_2(M)\ge\cdots$.

Sequences $(\xi_j)_j$ are understood to be sequences of iid copies of the variable $\xi$ from \Cref{thm:main}.
We write $\P_\Bg$ for a probability measure under which the coefficients $\xi_j$ in \eqref{model:YZ} are standard real Gaussians, and write $\E_\Bg$ for the associated expectation. (This notation is only used for comparisons of random variables in law -- we do not consider couplings.)


\subsection{Acknowledgements}

We thank 
Pavel Bleher, Yen Do, Oanh Nguyen, Oren Yakir and Ofer Zeitouni for helpful discussions and comments, and to Yakir and Zeitouni for showing us an early draft of their work \cite{YaZe} on the Gaussian case.
This project was initiated at the American Institute of Mathematics meeting ``Zeros of random polynomials'' in August 2019, where Bleher and Zeitouni were also participants. 
In particular, the idea used here and in \cite{YaZe} to study local linearizations emerged from those discussions.
We thank the workshop organizers and the Institute for providing a stimulating research environment.

\section{Preliminary reductions} 
\label{sec:prelim}

Our main objective in this section is to reduce our task to proving \Cref{prop:joint} below, which gives a comparison principle for the joint distribution of low-lying values for a discretized process over the circle.
Along the way we recall elements of the proof from \cite{YaZe} that we will need. For completeness we also include a brief description of their argument for the Gaussian case.

\subsection{Passage to local linearizations}
\label{sec:linearize}

We begin by recalling the approach from \cite{YaZe} for selecting near-local-minimizers of $|P_n(x)|$ on a discrete set; we refer to \Cref{sec:highlevel} for the high-level motivation of this approach. 
The criterion for $x_\al$ to be such a representative point is in terms of the local linearization $F_\al$ of $P_n$ at $x_\al$ -- the intuition is that for the mesh point $x_\al$ that is closest to a local minimizer of $|P_n(x)|$, it will also be close to the minimizer of $|F_\al(x)|$.
A key take-away from this approximation is that all information on near-minimizers of $|P_n(x)|$ is encoded in the values of $P_n$ \emph{and its derivative} at the mesh points.

We collect some notation and lemmas from \cite{YaZe}, with some minor modifications. 
Let $\Cp>4$ be a sufficiently large constant and set 
\begin{equation}	\label{def:N}
N:=\bigg\lfloor \frac{n^2}{\log^{\Cp} n} \bigg\rfloor.
\end{equation}
We divide $[-\pi,\pi]$ into $N$ intervals: letting
\[
x_\al = \frac{2\pi \al}{N}\,, \qquad\al =1,\dots, N,
\] 
we decompose
\[
[-\pi, \pi] = \bigcup_{\al=1}^N I_\al,
\quad\text{ where }  I_\al = \Big[x_\al- \frac{\pi}{N}, x_\al + \frac{\pi}{N}\Big].
\]
Note that for the case of real coefficients it suffices to consider $x_\al\in [0,\pi]$. 


Define 
\begin{align}	\label{def:YZ}
Y_\al &:= -\frac{\Re ( \pol_n(x_\al) \wb{ \pol_n'(x_\al)})}{| \pol_n'(x_\al)|^2}  \,,
\qquad
Z_\al :=   n\frac{\Im (\pol_n(x_\al) \wb{ \pol_n'(x_\al)})}{|\pol_n'(x_\al)|} \,.
\end{align}
We denote the local linearizations of $\pol_n$ given by
\begin{equation}
F_\al(x):= \pol_n(x_\al) + (x-x_\al) \pol_n'(x_\al).
\end{equation}
As shown in \cite[Section 1.3]{YaZe}, $|F_\al(x)|$ is minimized at $x=x_\al+Y_\al$, where it takes the value $|Z_\al|/n$; thus
\begin{equation}	\label{obs:YZ}
|F_\al(x_\al+Y_\al)| = |Z_\al|/n = \min_{x\in \R} |F_\al(x)|.
\end{equation}
(The sign is kept on $Z_\al$ only for convenience -- we mention that the sign encodes whether the origin is to the left or right of the curve $\{P_n(x): x\in [-\pi,\pi]\}$ as $x$ increases through $x_\al$, but this fact will not be used.)

We denote the $2\pi n$-periodic trigonometric polynomial 
\begin{equation}	\label{def:ptil}
\tpol_n(s) = \pol_n(s/n),\qquad s\in \R.
\end{equation}
This scaling will often be convenient since all of its derivatives are typically of order 1. 

We consider the collection $\{Z_\al\}_{\al\in [N]}$ as a point process on $\R$.
The scaling by $n$ means we focus on (signed) low-lying values of $|P_n|$.
Now we give the criterion by which ``representative'' near-minimizers are selected.
Let  $\CA_\al :=  \CA_\al' \cap \CA_\al''$ where 
\[
\CA_\al' :=\{|Y_\al | \le \pi/N, |Z_\al| \le \log n \}
\]
and
\begin{align*}
\CA_\al'' :&=  \{|\pol_n(x_\al)| \le n^{-1/2}, |\pol_n'(x_\al)|\in [ n\log^{-\Cp/2} n, C_0  n\sqrt{\log n}] \}\,,
\end{align*}
and define the point process
\begin{equation}\label{eqn:CM}
\Points_n = \sum_{\al=1}^N \delta_{X_\al},\qquad
X_\al:= Z_\al \1_{\CA_\al} + \infty \1_{\CA_\al^c}\,.
\end{equation}
The event $\CA_\al'$ is the condition on the local linearization that was described above, while $\CA_\al''$ enforces some regularity of $P_n$ on $I_\al$.

The following control on the second derivative will be used to show that the local linearizations $F_\al$ are good approximations to $\pol_n$ at the scale of the intervals $I_\al$. 

\begin{lemma}[Derivative bounds]
\label{lem:derivatives}
For $\dbd>1$ and integer $k\ge 0$ let $\CG_{k}(\dbd)$ be the event that 
\[
\sup_{s\in \R } |\tpol_n^{(k)}(s)|=\frac1{n^k}\sup_{x\in [-\pi,\pi] }|P_n^{(k)}(x)|  \le  \log^\dbd n.
\]
There exists $c=c(k)>0$ depending only on $k$ and the sub-Gaussian moment of $\xi$ such that
\[
\P(\CG_{k}(\dbd)^c) \le \exp(-c \log^{2\dbd} n).
\]
\end{lemma}

\begin{proof}
Fix $K$ and $k$.
It suffices to show the claimed bound for  $R:= \Re\tpol_n^{(k)}$. 
By Bernstein's inequality, 
\[
\sup_{t\in [-n\pi, n\pi]} |R'(t)| \ll  \sup_{t\in [-n\pi, n\pi]} |R(t)|\,,
\]
so if we assume that $\sup_t |R(t)|$ is attained at $t_0$, then for all $|t-t_0| \le c_0$ for a sufficiently small constant $c>0$, we have 
\[
|R(t)| \ge |R(t_0)| - |t-t_0| \sup_{t\in [-n\pi, n\pi]} |R'(t)|  > |R(t_0)|/2.
\]
It follows that if we divide $[-n\pi, n\pi]$ into $O(n)$ intervals $J_i$ of sufficiently small length and with midpoints $t_i$, then  we have $\sup_i |R(t_i)| > \frac{1}{2}\sup_{t\in [-n\pi, n\pi]} |R(t)|$. Hence
\begin{align*}
\P(\sup_{t\in [-n\pi, n\pi]} |R(t)| \ge  (\log n)^{\dbd} ) \le \sum_i \P(|R(t_i)| 
&\ge  (\log n)^{\dbd}/2) \\
&\ll n \exp(-c' (\log n)^{2\dbd}) \le \exp(-c (\log n)^{2\dbd})\,,
\end{align*}
where we used a sub-Gaussian tail estimate for the upper bound for each $t_i$.
\end{proof}

The next proposition shows that near-minimizers are typically well separated. 
The proof is a straightforward modification of the proof of \cite[Lemma 2.11]{YaZe} and is deferred to \Cref{app:neararcs}. There is the minor issue that a local minimizer for $\pol_n$ may cause a low value for two neighboring linearizations simultaneously, as accounted for in part (i). This will (unfortunately) present some issues of a purely technical nature in the proof of \Cref{prop:moments} below.

\begin{lemma}
\label{lem:neararcs}
On the event $\CG_{2}(\Cp/2)$ we have
\begin{enumerate}[(i)]
\item If $\CA_{\al}$ and $\CA_{\al +1}$ hold, then 
\[
Y_\al \in [\frac{\pi}{N} - \frac{\pi} {N \log^{\Cp/4} n}, \frac{\pi}{N}].
\]
\item Furthermore, $\CA_\al$ and $\CA_{\al'}$ cannot hold simultaneously as long as 
\[
2\le |\al'-\al| \le \frac{n}{\log^{3\Cp}n}.
\]
\end{enumerate}
\end{lemma}


\subsection{The Yakir--Zeitouni invariance argument}  
\label{sec:YZ}

Now we discuss briefly the key remaining ideas of \cite{YaZe} for the Gaussian case (or the case with small Gaussian component as in \eqref{eqn:YZ:xi}), which employs a strategy used by Biskup and Louidor in their work on extreme values of the planar discrete Gaussian free field \cite{BiLo} .
The approach combines the following ingredients:
\begin{enumerate}
\item A Gaussian computation showing that for any interval $[a,b]\subset \R$ we have $\lim_{n \to \infty} \E (\CM_n([a,b])) = \sqrt{\frac{\pi}{3}} (b-a)$.
\item A consequence of a general result of Liggett \cite{Liggett}:\footnote{\revised{For the interested reader, we note that a new proof of Liggett's general result in a special case sufficient for this application was recently obtained in \cite{CGS:Liggett}.}} that if the law of a point process is invariant under adding an independent Gaussian perturbation to each point, then it is a Poisson point process of constant intensity.
\item A consequence of the Gaussianity of the field $\{P_n(x)\}_{x\in [-\pi,\pi]}$: that if $Q_n$ is an independent copy of $P_n$, then $\wh{P_n}(x) = \sqrt{1-\frac{1}{n^2}} P_n(x) + \frac{1}{n} Q_n(x)$ is identically distributed to $P_n(x)$.
\item The fact that near-minimizers of $|P_n|$ are well separated (from a strengthening of  \Cref{lem:neararcs}).
\end{enumerate}
Roughly speaking, from (3) one can view $\wh{P_n}$ as a perturbation of $P_n$ by an independent Gaussian field $\frac1nQ_n$ of typical size $1/n$, which is the scale of the minimum modulus. Thus, the point process $\wh{\CM}_n$ is obtained from $\CM_n$ by (a slight rescaling and) a perturbation of each point by a standard Gaussian. Now from (4), the low values of $|P_n(x)|$ occur at points $x$ that are sufficiently separated that (as one can show) the values of $Q_n(x)$ at these near-minimizers are nearly uncorrelated. Hence, the point process $\wh{\CM}_n$ is approximately a point process obtained from $\CM_n$ by perturbing each $X_\al$ by an independent Gaussian. From (2) we get that $\wh{\CM}_n$, and hence, $\CM_n$, is a Poisson point process of constant intensity, and from (1) it follows that the intensity is $\sqrt{\pi/3}$. 
(To apply (2) one cannot actually argue at finite $n$ as just described, but instead one needs to pass to subsequential limiting point processes, obtained from the tightness implied by (1); in the end one finds a limiting Poisson point process of the same intensity regardless of the subsequence.)

Morally speaking, the exponential law is then a straightforward consequence of the minimum being approximately the smallest (absolute) value of a Poisson point process on $\R$. 
The formal argument requires some considerable work to justify all of the approximations, and the above sketch glides over many important points; we invite the reader to see \cite{YaZe} for further details.


\subsection{Towards universality: matching moments over smooth points}
\label{sec:moments}

It should be evident that the beautiful argument of \cite{YaZe} just described relies heavily and in several different ways on properties of the Gaussian distribution.
Towards establishing \Cref{thm:main}, our approach is to
 establish universality for the joint distribution of $X_{\al}$ at any fixed number of indices $\al\in [N]$ (in particular this yields universality for the joint intensity functions of the point process $\CM_n$).
From this one can deduce universality of moments $\E(\CM_n([-\tau,\tau])^\mom)$ of all order, leading to universality for the distribution function $\P(m_n\le \tau/n)$.



For general $\xi$, the main difficulty for studying the joint distribution of $P_n(x_i)$ and its derivative at $m$ different points $x_i$, or even at a single point $x$, is that the distribution is highly dependent on arithmetic properties of the points. 
Consider the case of Rademacher coefficients. 
At $x=0$ we have $P_n(0)=\frac1{\sqrt{2n+1}}\sum_{j=-n}^n\xi_j$ -- while from the Central Limit Theorem this approaches the $\CN_\R(0,1)$ distribution, it does so at the slowest possible rate, and the distribution is only smooth (i.e. comparable to Lebesgue measure on balls of radius $\delta$) at scales $\delta$ much larger than $1/\sqrt{n}$.
At $x =\pi/2$ we have that $P_n(\pi/2)$ splits into independent real and imaginary sums, each tending to the $\CN_\R(0,1/2)$ distribution at the slowest possible rate. 
The situation is slightly improved at $x=\pi/4$, for which one can obtain a meaningful small ball estimate at scale $\delta\sim 1/n$ with some effort. 
As we shrink the scale $\delta$ at which we desire $P_n(x)$ to have an effectively smooth distribution, the collection of ``structured'' angles that we must avoid increases. 

Thus we see that Diophantine approximation will play a crucial role in our arguments. 
Indeed, such considerations played a strong role in the argument of Konyagin and Schlag for the upper bound \eqref{KS:UB}. 
That work only dealt with the field at single points, however; to compare the joint distribution of $P_n$ and its derivative at an arbitrary fixed number of points we need finer control. 

We quantify the level of approximability of points $x$ by rationals as follows:

\begin{definition}[Smooth points]
\label{def:smooth}
For $\Smooth>0$, we say a point $t\in \R$ is \emph{$\Smooth$-smooth} if 
\[
\Big\| \frac{p_0 t}{\pi n}\Big\|_{\R/\Z} > \frac{\Smooth}{n}\qquad \forall\; p_0\in \Z\cap [-\Smooth-1,\Smooth+1], p_0 \neq 0 .
\]
We say a tuple $(t_1,\dots, t_\mom)$ is $\Smooth$-smooth if $t_r$ is $\Smooth$-smooth for each $1\le r\le \mom$.
\end{definition}

Thus in the special case that $K<1$ then  $t\in \R$ is \emph{$\Smooth$-smooth} if $\| \frac{t}{\pi n}\|_{\R/\Z} > \frac{\Smooth}{n}.$ Observe also that if $n^{-1+\kappa} \le \|\frac{t}{\pi n}\|_{\R/\Z} \le n^{-2\kappa}$ then $t$ is $n^{\kappa}$-smooth.

The following lets us focus on potential minimizers that are smooth. 

\begin{lemma}[Ruling out bad arcs]
\label{lem:badarcs}
For $\smooth>0$ let $\badarcs(\smooth)$ be the set of points $x\in \R$ such that $nx$ is not $n^\smooth$-smooth. 
There exist absolute constants $\smooth_0,c_0>0$ such that
\[
\P\big(\, \exists x\in \badarcs(\smooth_0): |\pol_n(x)| \le n^{-1+c_0} \,\big) = o(1).
\]
\end{lemma}


\begin{proof}
This follows from the argument for \cite[Lemma 3.3]{KSch}; one only needs two modifications:  
\begin{enumerate}
\item Whereas they considered $A$-smooth points for $A$ fixed, their bounds in fact allow $A$ to grow as fast as $n^{\smooth_0}$ for $\kappa_0$ sufficiently small. (One also notes that their parameter $\eps$ may grow as fast as $O(n^{3/4})$.)
\item Whereas their model takes the sum in \eqref{model:YZ} to run over $[0,n]$ rather than $[-n,n]$, they only need that the covariance matrix for $(\Re P_n(x), \Im P_n(x))$ has eigenvalues bounded below by $\gg n^2 \min( 1, |x|, |\pi-x|)^2$ for $\min(|x|, |\pi-x|) \gg n^{-1-c}$ for a small absolute constant $c>0$, which for the present model follows from display (2.21) in \cite{YaZe}. 
(One may alternatively apply the proof of \cite[Lemma 3.3]{KSch} but condition on the variables $(\xi_j)_{-n\le j<0}$ before applying the Berry--Esseen theorem.)
\end{enumerate}
\end{proof}

With $\smooth_0$ as in \Cref{lem:badarcs} we now consider the thinned point process
\begin{equation}	\label{def:Mn.sharp}
\Points_n^\sharp := \sum_{\al: x_\al \notin \badarcs(\smooth_0)} \delta_{X_\al}.
\end{equation}
\Cref{thm:main} will be deduced from the following comparison of moments. The proof is deferred to \Cref{sec:proof-moments}.

\begin{proposition}[Moment matching]
\label{prop:moments}
For any fixed $\height>0$ and integer $\mom\ge1$ we have 
\begin{equation}\label{eqn:moment}
\lim_{n\to \infty}\E \Big( \Points_n^\sharp\big([-\height,\height]\big)^\mom\Big) = \lim_{n\to \infty}\E_{\Bg} \Big( \Points_n^\sharp\big([-\height,\height]\big)^\mom\Big),
\end{equation}
where we recall that $\E_{\Bg}$ stands for expectation under the Gaussian model from Theorem \ref{thm:YZ}.
\end{proposition}

\subsection{Joint distribution over spread points}

Expanding the moments in \eqref{eqn:moment} leads to consideration of joint events that $X_{\alpha_i}$ is small at $m$ different points $x_{\alpha_i}$, $1\le i\le m$. 
In addition to the smoothness already imposed in the definition of $\Points_n^\sharp$,
we will require all of the points to be separated from one another, in the following sense:

\begin{definition}[Spread tuples]
\label{def:spread}
For $m\ge 2$ and $\spread>0$, we say $\bs{t}=(t_1,\dots, t_\mom)\in \R^\mom$ is \emph{$\spread$-spread} if
\[
\Big\| \frac{t_r\pm t_{r'}}{2\pi n} \Big\|_{\R/\Z} \ge \frac{\spread}{n}
\qquad \forall \, 1\le r<r' \le \mom \mbox{ (and all choices of the signs $\pm$). }
\]
 For $m=1$, we say that $\bs t = t\in \R$ is  \emph{$\spread$-spread} if
\[
\Big\| \frac{t}{2\pi n} \Big\|_{\R/\Z} \ge \frac{\spread}{n}.
\]
\end{definition}
It is remarked that in the definition above we prevent $t_r$ from being close to $t_{r'}$ and $-t_{r'}$ at the same time, and this condition is necessary to hope for asymptotically independence between $P_n(t_r)$ and $P_n(t_{r'})$, especially in the case that $\xi$ is real-valued. 



In what follows we denote
\begin{equation}
s_\al:= nx_\al, \qquad \al\in [N].
\end{equation}
Recalling the scaled polynomial $\til P$ from \eqref{def:ptil}, we have
\begin{equation}	\label{def:YZ.til}
Y_\al =  - \frac{1}{n} \frac{\Re (\til P_n(s_\al) \wb{\til P_n'(s_\al)})}{|\til P_n'(s_\al)|^2} 
\qquad
Z_\al =   n \frac{\Im (\til P_n(s_\al) \wb{\til P_n'(s_\al)})}{|\til P_n'(s_\al)|}.
\end{equation}
The main step towards the proof of \Cref{prop:moments} is the following:

\begin{proposition}
\label{prop:joint}
Fix an $m$-tuple of indices $(\al_1,\dots, \al_\mom)\in [N]^\mom$. 
Assume for some $\smooth>0$ that $s_{\al_1}, \dots, s_{\al_\mom}$ are $n^\smooth$-smooth
and that $\bs s= (s_{\alo}, \dots, s_{\alm})$ is $1$-spread.
Then for any $\height>0$, 
\[
\bigg|\, \pr\bigg(\bigwedge_{i\in [\mom]}  |X_\ali| \le \height \bigg) -\pr_{\Bg}\bigg( \bigwedge_{i\in [m]} |X_\ali|\le \height \bigg) \,\bigg| = o(N^{-\mom}),
\]
where the rate of convergence depends on $\mom,\height,\smooth,$ and $\Cp$.
\end{proposition}

We prove \Cref{prop:joint} in \Cref{sec:RW} below, where we convert the task to a problem involving a random walk in $\R^{4m}$. Before proceeding we collect the following useful property of a smooth $m$-tuples, which basically says that we can simultaneously dilate the points $t_r$ to be well separated on the torus.
This result will be useful for the proof of \Cref{lem:cov} below for showing that the distribution of an associated random walk is genuinely full-dimensional, and also for Section \ref{sec:char} when we bound $\prod_{r=1}^{\mom-1} \|\frac{L(t_\mom \pm t_{r})}{2\pi n}\|_{\R/\Z}$ from below for some $L$. 

\begin{lemma}
\label{lem:dilate}
Assume $(t_1,\dots, t_\mom)\in \R^\mom$ is $\spread$-spread for some $\spread>0$, and let $\spread\le K=o(n)$. 
There exists an integer $\dil \asymp n/K$ such that
\begin{equation}	\label{LB.dil}
\Big\| \frac{\dil\cdot(t_r\pm t_{r'})}{2\pi n} \Big\|_{\R/\Z} \gg_\mom \spread/K
\qquad \forall 1\le r<r'\le m
\end{equation}
(and all choices of the signs).
In particular, if $(t_1,\dots, t_\mom)$ is $\omega(1)$-spread then there exists $\dil\le n$ such that 
\begin{equation}
\Big\| \frac{\dil\cdot(t_r\pm t_{r'})}{2\pi n} \Big\|_{\R/\Z}  \gg_\mom 1
\qquad \forall 1\le r<r'\le m.
\end{equation}

In case $m=1$ then there exists an integer $\dil \asymp n/K$ such that $\| \frac{\dil \cdot t}{2\pi n} \|_{\R/\Z} \gg_\mom \spread/K.$
\end{lemma}


\begin{proof}[Proof] The case $m=1$ is clear, so we just need to focus on $m\ge 2$.
Assume towards a contradiction that there exists $\eps=\eps(\mom)>0$ such that for every $j\in [n/2K, n/K]$ there exists a pair of distinct indices $r,r'\in [m]$ such that 
\begin{equation}	\label{cov.bd1}
\min\bigg\{ \bigg\|\frac{j(t_r-t_{r'})}{2\pi n}\bigg\|_{\R/\Z}\, ,\, 
\bigg\|\frac{j(t_r+t_{r'})}{2\pi n}\bigg\|_{\R/\Z}\bigg\}
 \le \eps \lam/K
\end{equation}
By pigeonholing, there is a pair of distinct indices $r,r'\in [\mom]$ and subset $J\subset [n/2K, n/K]$ of size $\gg n/K\mom^2$ such that either the first quantity in the minimum in \eqref{cov.bd1} is bounded by $\eps\lam/K$ for all $j\in J$, or the second is bounded by $\eps\lam/K$ for all $j\in J$. 
We focus on the former case; the latter is handled by a similar argument.

As $|J|$ is of the same order as its diameter, there exists $C=O_\mom(1)$ so that $CJ-CJ$ contains a homogeneous arithmetic progression of length $\gg  n/K$ (see for instance \cite[Lemma B.3]{Tao:freiman}). 

\begin{claim}\label{claim:dividing} 
Assume that $z=e^{i \theta}, |\theta|\le \pi/8 $ such that for all $1\le \ell \le M$ we have $|1-z^{\ell}| \le 1/32$ for a sufficiently large $M$. Then $|\theta|=O(1/M)$.
\end{claim}

\begin{proof}
By assumption, $|\theta| \le \pi/8$ and $\|2^k\theta\|_{\R/\Z} \le \pi/8$ for all $1\le k\le \log M$, and so we can repeatedly estimate $|\theta|$ to obtain $|\theta| =O(1/M)$. 
\end{proof}

By the triangle inequality, for $\eps$ sufficiently small depending on $C$, by \Cref{claim:dividing} this would imply there exists $C_{r,r'}=O_\mom(1)$ such that 
\begin{equation}	\label{MMr}
\bigg\|\frac{C_{r,r'}(t_r-t_{r'})}{2\pi n}\bigg\|_{\R/\Z} \ll_\mom \frac{\eps \lam/K}{ n/K} \ll_\mom \eps\lam/n.
\end{equation}
Let $\CN_1$ be the collection of all pairs $(r,r')$ such that \eqref{MMr} holds, taking $C_{r,r'}$ to be the smallest such positive integer.
We have shown that $\CN_1$ is nonempty. 
By the assumption that $\bs t$ is $\lam$-spread we have that $C_{r,r'}>1$ for all $(r,r')\in \CN_1$. 

\begin{claim}
Assume that for some $x\in \R,\delta>0$ and positive integer $M$ we have $\|x\|_{\R/\Z}>\delta$ and $\|Mx\|_{\R/\Z}\le \delta$. Then
\[
\|x\|_{\R/\Z} > 1/2M.
\]
\end{claim}

\begin{proof}
Assuming otherwise, we have $\|Mx\|_{\R/\Z} = M\|x\|_{\R/\Z}  >M\delta$, a contradiction. 
\end{proof}

From the above claim, \eqref{MMr}, and the assumption $\bs t$ is $\lam$-spread, it follows that if $\eps$ is sufficiently small, then
\[
\bigg\|\frac{t_r-t_{r'}}{2\pi n} \bigg\|_{\R/\Z} \ge 1/2C_{r,r'}
\]
for each $(r,r')\in \CN_1$. 
Set $D_1 = \prod_{(r,r')\in \CN_1} C_{r,r'} = O_\mom(1)$,
and let $I_1$ be intersection of the progression $\{1+\ell D_1\}_{\ell\in \Z}$ with $[n/2K, n/K]$.
Applying the triangle inequality, if $L=1+lD_1 \in I_1$ then for all $(r,r')\in \CN_1$,
\begin{align*}
\bigg\|\frac{L (t_r-t_{r'})}{\pi n}\bigg\|_{\R/\Z} =\bigg\|\frac{(1+lD_1) (t_r-t_{r'})}{2\pi n}\bigg\|_{\R/\Z} &\ge \bigg\|\frac{t_r-t_{r'}}{2\pi n}\bigg\|_{\R/\Z} -\bigg\|\frac{l\frac{D_1}{C_{r,r'}} C_{r,r'}(t_r-t_{r'})}{2\pi n}\bigg\|_{\R/\Z}  \\
& \ge  1/2C_{r,r'} - (n/K) O_m(\eps\lam/n) \ge \eps\lam /K
\end{align*}
 provided that $\eps $ is sufficiently small.
 Now if no $L \in I_1$ satisfies the conclusion of our lemma, then for each $L \in I_1$ there is a pair $(r,r') \notin \CN_1$ that violates the condition, and then we repeat the above process, with $\CN_2$ being the collection of such pairs. Set $D_2= \prod_{(r,r')\in \CN_2} C_{r,r'}$ (and so $D_2=O_\mom(1)$) and let $I_2$  be intersection of the progression $\{1+\ell D_1D_2\}_{\ell\in \Z}$ with $[n/2K, n/K]$, we then continue the process as above. As each time we get rid of at least one pair $(t_r,t_{r'})$, the process for differences terminates after $\binom{\mom}{2}$ steps with $\Theta(n/K)$ indices left to choose. Finally, we can start the process for $t_r+t_{r'}$ with $j$ (appearing in \eqref{MMr}) chosen from these indices; the remaining iterations are identical as above.\end{proof}

\section{Random walk in phase space}
\label{sec:RW}

The key ingredients for the proof of \Cref{prop:joint} are local small ball estimates and a comparison principle for an associated random walk in $\R^{4\mom}$, which we now define.

For a fixed tuple $\bs t= (t_1,\dots, t_\mom)\in \R^\mom$ and $j\in \Z$ we denote the vectors 
\begin{align}
\avec_j = \avec_{j}(\bs t) 
&:= \big( \sin(jt_1/n),\dots, \sin(jt_\mom/n) \big) \;\in \R^\mom	\notag\\
\bvec_j= \bvec_{j}(\bs t) 
&:= \big( \cos(jt_1/n),\dots, \cos(jt_\mom/n) \big) \;\in \R^\mom	\notag
\end{align}
and
\begin{equation}
\qquad\wvec_j = \wvec_j(\bs t)=\big(\avec_j\,,\,(j/n)\bvec_j\,,\,\bvec_j\,,\, -(j/n)\avec_j \big) \;\in \R^{4m}\,.	\label{def:wj}
\end{equation}
For a finite set $J\subset\Z$ we let 
$W_{J}= W_{J}(\bs t)$ 
be the 
\revised{$|J|\times (4m)$}
matrix with rows $\wvec_j$, $j\in J$.
Note that $\bs w_j$ gives the values of the functions $\sin(\frac{j}n\,\cdot), \cos(\frac{j}n\,\cdot)$ and their derivatives at the points $t_1,\dots, t_\mom$. 
We consider the random walk
\begin{equation}	\label{def:Snt}
S_n(\bs t) := \sum_{j=-n}^n \xi_j \bs w_j(\bs t) = W_{[-n,n]}^\tran \bxi \;\in\, \R^{4m}
\end{equation}
with $\bxi=(\xi_j)_{j\in [-n,n]}$ a vector of iid copies of a real-valued $\xi$.

\subsection{Control on the characteristic function}

The following is the key technical ingredient for controlling the distribution of the random walks $S_n(\bs t)$.

\begin{theorem}
\label{thm:char}
Let $\bs t=(t_1,\dots,t_\mom)\in \R^m$ be $n^\smooth$-smooth 
and $\spread$-spread for some $\smooth\in (0,1)$ and $\omega(n^{-1/8\mom})\le \spread\le1$. 
Then for 
any fixed $K_*<\infty$ and any $\bs x\in \R^{4m}$ with $n^{-1/8}\le \|\bs x\|_2\le n^{K_*}$,
\[
|\E e(\langle S_n(\bs t),\bs x\rangle)| \le \exp( - \log^2n )
\]
for all $n$ sufficiently large depending on $K_*, m, \kappa,$ and the sub-Gaussian constant for $\xi$.
\end{theorem}

We note that here the sub-Gaussianity hypothesis enters only to have a uniform anti-concentration bound for $\xi$ and could be replaced by a bound on the L\'evy concentration function.

We defer the proof of this theorem to \Cref{sec:char}. 
Now we state the two main consequences of \Cref{thm:char} towards the proof of \Cref{thm:main}.
By combining \Cref{thm:char} with an Edgworth expansion, we will obtain the following quantitative comparison with the Gaussian model. 
In the following we write 
$\Gamma=\Gamma_n(\bs t)\in \R^{4m}$ for a Gaussian vector with covariance matrix $\frac1{2n+1}W_{[-n,n]}^\tran W_{[-n,n]}$.
Note that this is the distribution of $\frac1{\sqrt{2n+1}}S_n(\bs t)$ with iid standard real Gaussians in place of $\xi_j$.

\revised{


\begin{theorem}
\label{thm:box}
Let $\bs t= (t_1,\dots, t_\mom)$ be $n^\smooth$-smooth and $1$-spread for some $\smooth>0$. 
Fix $K>0$ and let $Q\subset\R^{4k}$ be a box (cartesian product of intervals) with side lengths at least $n^{-K}$.
Then
\[
\revised{\sup_{w\in\R^{4m}}}
\Big| \P\Big(   \frac1{\sqrt{2n+1}}S_n(\bs t)  \in  Q\Big) 
- \P\big(  \Gamma_n(\bs t) \in Q \big) \Big| 
\ll n^{-1/2} |Q|
\]
where $|Q|$ is the volume of $Q$, and the implied constant depends only on $m, \smooth, K$, 
and the sub-Gaussian constant for $\xi$. 
\end{theorem}
}

\begin{remark}
The proof shows that in place of the sub-Gaussianity assumption we only need that $\xi$ has $O(\mom)$ finite moments.
\end{remark}

We defer the proof of \Cref{thm:box} to \Cref{sec:box}.

By standard arguments, the control on the characteristic function of $S_n(\bs t)$ provided by \Cref{thm:char} yields an optimal small ball estimate at arbitrary polynomial scales:

\begin{theorem}[Small ball estimate]
\label{thm:smallball}
With $\bs t= (t_1,\dots, t_\mom)$ as in \Cref{thm:char}, for any $K<\infty$ and any $\delta\ge n^{-K}$, 
\[
\sup_{w\in \R^{4m}} \pr \bigg( \frac1{\sqrt{2n+1}} S_n(\bs t) \in B(w, \delta) \bigg) = O_{\mom,\smooth,K}(\lam^{-3\mom}\delta^{4\mom}).
\]
\end{theorem}

The proof of \Cref{thm:smallball} is deferred to \Cref{sec:smallball}.
We note the following consequence, giving anti-concentration for the polynomial $P_n$
\revised{(recall the rescaled polynomial $\til P_n$ from \eqref{def:ptil})}.

\begin{corollary}[Small ball estimate for polynomials]
\label{cor:smallball}
Assume that $t$ is $n^\smooth$-smooth. Then for any $K>0$ and $\delta\in [n^{-K},1]$,
$$\P( |\revised{\til P}_{n}(t/n)|\le \delta) = O_{\kappa,K}(\delta^2) 
\qquad\text{ and }\qquad
 \P( |\revised{\til P}'_{n}(t/n)|\le \delta) = O_{\kappa,K}(\delta^2) .$$
\end{corollary}


\subsection{Non-degeneracy of the covariance matrix}


As a first step towards controlling the distribution of $S_n(\bs t)$ we need to show that the random walk is genuinely $4m$-dimensional, which amounts to showing the covariance matrix $W_{[-n,n]}^\tran W_{[-n,n]}$ has smallest singular value of order $n$. 
This is accomplished by the following lemma, under the (necessary) assumption that the points $t_1,\dots, t_\mom$ are spread.

\begin{lemma}
\label{lem:cov}
Let $J\subset [n]$ be an interval with $|J|\gg n$.
If $\bs t=(t_1,\dots, t_\mom)\in \R^m$ is $\spread$-spread for some $\spread>0$, then 
\[
\|W_{J}(\bs t) u\|_2^2 \gg_\mom \min(\spread,1)^{6m-3} n
\]
uniformly over unit vectors $u\in S^{4m-1}$. 
\end{lemma}


\begin{remark}
We note that for the case $\xi_j\sim \CN_\R(0,1)$, the above control on the covariance matrix is enough to deduce an optimal small ball estimate at all scales. For general distributions we need \Cref{thm:char}, the proof of which amounts to showing that for $v$ of size $n^{O(1)}$, the vector $W_J(\bs t)v$ avoid the \emph{lattice} $\Z^n$, rather than just the origin as above. 
\revised{The proof below can be read as a warmup to the more technical proof of \Cref{thm:char}, where a similar (but more complicated) differencing strategy is used.}
\end{remark}

\revised{
\begin{remark}
We point out that if $\lambda$ is growing with $n$, it is not hard to show by computations similar to \cite[Lemma 3.2]{KSch} that $\frac1{|J|}W_J(\bs t)^\tran W_J(\bs t)$ asymptotically splits into $m$ well-conditioned blocks. However, when $\lambda$ is bounded or shrinking with $n$ the covariance matrix becomes increasingly degenerate. We note that \cite[Lemma 3.2]{KSch} also contains estimates for the covariance matrix of the real and imaginary parts of $P_n(t)$ at a single point $t$ that is only $n^{-1/2}$-spread. In principle it should be possible to extend those arguments to the above setting with $m>1$ and additional columns for $P_n'$; however, this appears to involve technical case analysis, and in the end we do not think it would lead to a significantly shorter proof than the one given below. 
\end{remark}
}

\begin{proof}[Proof of \Cref{lem:cov}]
Without loss of generality we may assume $\spread\in (0,1)$.
Fix a vector $u = (u^1,u^2,u^3,u^4)\in S^{4\mom-1}$.
The $j$th entry of 
\revised{$W_J(\bs t)u$} is
\[
\langle \bs w_j,u\rangle =\sum_{r=1}^\mom u^1_r \sin(jt_r/n) + u^2_r (j/n)\cos(jt_r/n) + u^3_r \cos(jt_r/n) - u^4_r (j/n)\sin(jt_r/n).
\]
Substituting $\cos(jt_r/n) = \frac12(e_n(jt_r) +e_n(-jt_r))$ and $\sin(jt_r/n) = -\frac{\ii}{2} (e_n(jt_r) - e_n(-jt_r))$, 
the above becomes
\begin{align*}
&\frac12\sum_{r=1}^m (u_r^3-\ii u_r^1) e_n(jt_r) + (u_r^3+ \ii u_r^1)e_n(-jt_r)\\
&\qquad\qquad\qquad\qquad + (u_r^2+\ii u_r^4)(j/n) e_n(jt_r) + (u_r^2-\ii u_r^4) (j/n) e_n(-jt_r)\\
&\qquad=\Big\langle \big(\bs e_j  , \bar{\bs e}_j, (j/n) \bs e_j, (j/n) \bar{\bs e}_j\big)\,,\, A u \Big\rangle
\end{align*}
where
\[
\bs e_j := ( e_n(jt_1), \dots, e_n(jt_\mom))
\]
and
\[
A = \frac12
\begin{pmatrix} 
-\ii I_\mom & 0 & I_\mom & 0 \\
\ii I_\mom & 0 & I_\mom & 0\\
0 & I_\mom & 0 & \ii I_\mom \\
0 & I_\mom & 0 & -\ii I_\mom
\end{pmatrix}
\]
where $I_\mom$ is the $\mom\times\mom$ identity matrix and 0 is the square  matrix of 0s. 
Since $\|A^{-1}\|= O(1)$, it suffices to show
\[
\|M v\|_2^2 \gg_\mom \lam^{6m-3} n
\]
uniformly for $v$ in the complex sphere  $S_\C^{4\mom-1}$, where $M\in \C^{n\times 4m}$ is the matrix with rows 
\[
(\bs e_j, \bar{\bs e}_j, \ii(j/n) \bs e_j, \ii (j/n) \bar{\bs e}_j).
\]
From \Cref{lem:dilate} there exists an integer $\dil$ with $n\ll_m\dil<n/100m$ such that
\[
\Big\| \frac{\dil\cdot(t_r\pm t_{r'})}{2\pi n} \Big\|_{\R/\Z} \gg_m \spread
\qquad \forall 1\le r<r'\le m.
\]

For notational convenience we will consider $M$ with rows of the general form 
\[
(e_n(jt_1),\dots, e_n(jt_d), \ii (j/n) e_n(jt_1),\dots, \ii (j/n) e_n(t_d))
\]
satisfying
\begin{equation}	\label{s.dilated}
\Big\| \frac{\dil\cdot(t_r-t_{r'})}{2\pi n} \Big\|_{\R/\Z} \ge \spread_0
\qquad \forall 1\le r<r'\le d
\end{equation}
for some $\spread_0\in (0,1)$ and $n\ll_d L<n/50d$, and aim to show
\begin{equation}	\label{cov.goal1}
\inf_{v\in S_\C^{2d-1}}\|M v\|_2^2 \gg_d \lam_0^{3d-3} n.
\end{equation}
One passes back to the previous case by taking $d=2\mom$ and $(t_1,\dots, t_{2\mom}) = (t_1,\dots, t_\mom, -t_1,\dots, -t_\mom)$, and substituting any $c(m)\lam$ for $\lam_0$.

Let $P$ denote the intersection of the interval $J$ with the progression $\{iL: i\in\Z\}$, and let $M_P$ denote the submatrix of $M$ with rows indexed by $P$. Note that $|P|\asymp_d 1$.
We will first show
\begin{equation}	\label{MS.goal0}
\inf_{v\in S_\C^{2d-1}}\| M_P v\|_2^2 \gg_d \lam_0^{2d-2}.
\end{equation}
To do this we consider the twisted second-order differencing operators of the form
\begin{equation}	\label{def:Dt0}
(D_{t_0} f)(j) := \sum_{a=0}^2 {2\choose a} (-1)^a e(-aLt_0) f(j+ aL)
\end{equation}
acting on sequences $f: P\to \C$, for various choices of the parameter $t_0\in \R$.
Let us denote
\[
f_t(j) = e_n(jt),\qquad g_t(j) = \ii(j/n) e_n(jt) = \partial_t f_t(j).
\]
For $t,t_0\in \R$ and any $j\in P$ with $j+2L\in P$, we have
\begin{equation}	\label{DFt}
(D_{t_0} f_t)(j) = e_n(jt) \sum_{a=0}^2 {2\choose a} (-1)^a e(aL(t-t_0)) = \big[1-e_n(L(t-t_0))\big]^2 f_t(j)
\end{equation}
and
\begin{align}
(D_{t_0} g_t)(j) 
&= \sum_{a=0}^2 {2\choose a} (-1)^a e(aL(t-t_0)) \ii \frac{j+aL}{n} e_n((j+aL)t)	\notag\\
&= \ii (j/n)(D_{t_0} f_t)(j)
+ e_n(jt) \bigg[ -2\ii \frac{L}n e_n(L(t-t_0)) + 2\ii \frac{L}n e_n(2L(t-t_0)) \bigg]	\notag\\
&= \big[1-e_n(L(t-t_0))\big]^2 g_t(j) - 2\ii \frac{L}n e_n(L(t-t_0)) \big[1-e_n(L(t-t_0))\big] f_t(j)	\notag\\
&= \big[1-e_n(L(t-t_0))\big]^2 \big[ g_t(j) + \beta_L(t-t_0) f_t(j) \big]	\label{DGt}
\end{align}
where we write $\beta_L(s) := -2\ii\frac Ln e_n(Ls) / \big[1-e_n(Ls)\big]$. 
In particular, we have
\begin{equation}	\label{Dt.nihil}
(D_{t_0} f_{t_0})(j) = (D_{t_0}g_{t_0})(j) = 0\qquad \forall j.
\end{equation}
The key point about the factors $1-e_n(L\,\cdot)$ and  $\beta_L(\cdot)$ is that they are independent of $j$ and hence pass through the difference operators $D_{t_0}$. 

For the lower bound \eqref{MS.goal0} we partition the sphere into $d$ pieces
\[
S_r = \{ v\in S_\C^{2d-1}: |v_r|^2 + |v_{r+d}|^2 \ge 1/d\}\,\qquad 1\le r\le d
\]
and prove the bound separately on each piece. 
By symmetry it suffices to treat $S_d$.
We abbreviate
\[
G:= \prod_{r=1}^{d-1} \big[1-e_n(L(t_d-t_r))\big]^2\,, \qquad 
H:= \sum_{r=1}^{d-1} \beta_L(t_d-t_r).
\]
Iterating the identities \eqref{DFt}--\eqref{Dt.nihil}, we obtain that for any $j\in P$ such that $j+2dL\in P$, 
\[
\big( D_{t_1}\circ \cdots \circ D_{t_{d-1}} f_{t_r} \big) (j)  = 0 \qquad 1\le r\le d-1
\]
and otherwise
\[
\big( D_{t_1}\circ \cdots \circ D_{t_{d-1}} f_{t_d} \big) (j) 
=G \cdot f_{t_d}(j) .
\]
Similarly,
\[
\big( D_{t_1}\circ \cdots \circ D_{t_{d-1}} g_{t_r} \big) (j)  = 0 \qquad 1\le r\le d-1
\]
and otherwise 
\[
\big( D_{t_1}\circ \cdots \circ D_{t_{d-1}} g_{t_d} \big) (j) 
= G \cdot \big(g_{t_d}(j) + H\cdot  f_{t_d}(j) \big).
\]

Fix an arbitrary $v\in S_d$. 
Recognizing the sequences $(f_{t_r}(j))_{j\in P}, (g_{t_r}(j))_{j\in P}$ as the $2d$ columns of $M_P$, we have
\[
(M_P v)_j = \sum_{r=1}^d v_r f_{t_r}(j) + v_{r+d} g_{t_r}(j).
\]
Letting $\bs D$ be the matrix associated to the linear operator  $D_{t_1}\circ \cdots \circ D_{t_{d-1}}$ on $\C^P$, we have
\begin{align*}
(\bs D M_P v)_j 
&= v_d G f_{t_d}(j) + v_{2d}  G (g_{t_d}(j) + H f_{t_d}(j))\\
&=G \cdot e_n(jt_d) \big[ v_d + \big( \ii (j/n) + H\big) v_{2d} \big]
\end{align*}
for each $j\in P$ such that $j+2dL\in P$.
Taking the modulus of each side and square-summing we obtain
\begin{align*}
\sum_{j\in P : j+2d L \in P} |(\bs D M_P v)_j |^2 
& = |G|^2 \sum_{j\in P : j+2d L \in P} \big| v_d + \big( \ii (j/n) + H\big) v_{2d} \big|^2.
\end{align*}
From \eqref{s.dilated} we have
\[
G \ge (c\lam_0)^{2d-2}, \qquad H= O(d/\lam_0).
\]
In particular, since $v_d, v_{2d}$ and $H$ are independent of $j$, and $|v_d|^2 + |v_{2d}|^2 \ge 1/d$, the sum on the right hand side of the previous display is at least $\gg |P|/d^2\gg_d1$, so 
\[
\sum_{j\in P : j+2d L \in P} |(\bs D M_P v)_j |^2 \gg_d \lam_0^{2d-2}.
\]
On the other hand, since the matrix $\bs D$ has $\ell_2(P)\to \ell_2(P)$ operator norm $O(d)$, the left hand side is bounded above by $\ll_d \|M_P v\|_2^2$,
and we obtain \eqref{MS.goal0} as desired. 

It only remains to prove \eqref{cov.goal1}.
Consider the submatrices $M_P, M_{1+P}, \dots, M_{n_0+ P}$
composed of rows indexed by the shifted progressions $P, 1+ P, \dots, n_0+P$, respectively. 
If $n_0<L$ then these submatrices are all disjoint. 
Moreover, letting $F$ denote the $2d$-dimensional diagonal matrix with diagonal entries $e_n(t_1),\dots, e_n(t_d), e_n(t_1),\dots, e_n(t_d)$, we note that $M_{k+P}$ and $M_PF^k$ differ by a matrix of norm $O_d(k/n)$ (as they only differing in the dilations by $\ii j/n$ in the last $d$ columns).
Since $F$ is unitary we have $\sigma_{2d}(M_P F^k) = \sigma_{2d}(M_P) \gg_d \lam_0^{d-1}$, and taking $n_0=c(d)\lam_0^{d-1}n$ for $c(d)>0$ sufficiently small depending on $d$, from the triangle inequality we obtain that $\sigma_{2d}(M_{k+P}) \gg_d \lam_0^{d-1}$ for all $1\le k\le n_0$. 
Since $1+P,\dots, n_0+P$ are disjoint, we conclude that for any fixed $v\in S_\C^{2d-1}$, 
\[
\|Mv\|_2^2 \ge \sum_{k=1}^{n_0} \|M_{k+P}v\|_2^2 \gg_d n_0\lam_0^{2d-2}\gg_d \lam_0^{3d-3} n
\]
giving \eqref{cov.goal1} as desired.
\end{proof}

\section{Proof of \Cref{prop:joint}}
\label{sec:joint.proof}

In this section we combine Theorems \ref{thm:box} and \ref{thm:smallball} to prove \Cref{prop:joint}.
In fact we will need the following more general result, which in particular establishes universality for the joint distribution of the recentered near-local minimizers $Y_{\al_i}$ and corresponding near-local minima $X_{\al_i}$. 

\begin{proposition}
\label{prop:joint.gen}
Fix an $m$-tuple of indices $(\al_1,\dots, \al_\mom)\in [N]^\mom$, 
and assume $\bs s= (s_{\alo}, \dots, s_{\alm})$ is $n^\smooth$-smooth and $1$-spread for some $\smooth>0$.
Let $J_1,\dots, J_\mom\subset\R$, $J'_1,\dots,J'_\mom\subseteq[-\pi,\pi]$ be arbitrary compact intervals with lengths in the range $[n^{-L_0}, n^{L_0}]$ for some $L_0>0$, and denote the event
\begin{equation}
\CE= \bigwedge_{i\in [\mom]} \big\{ X_\ali \in J_i, \,NY_\ali \in J_i' \big\}.
\end{equation}
We have
\begin{align}	
&
\big| \P(\CE) - \P_{\CN_{\R}(0,1)}(\CE) \big|
\ll_{\mom,\smooth, L_0}
\frac{\log^{O(m)}n}{n^{1/2}N^m}  \prod_{i=1}^m |J_i||J_i'|.		\label{joint.compare}
\end{align}
Moreover, if $\bs s$ is $n^\smooth$-smooth and $\spread$-spread for some $\omega(n^{-1/8\mom})\le \spread\le1$, then we have the upper bounds
\begin{equation}	\label{joint.ub}
\P(\CE) 
 \ll_{\mom,\smooth, L_0}
\frac{ \log^{O(m)}n}{\lam^{3\mom} N^{\mom}  }\prod_{i=1}^\mom |J_i||J_i'|\,,
\end{equation}
and
\begin{equation}	\label{joint.ub.gaussian}
\P_{\CN_\R(0,1)}(\CE) 
 \ll_{\mom,\smooth, L_0}
\frac{1}{\lam^{O(\mom^2)} N^{\mom}  }\prod_{i=1}^\mom |J_i||J_i'|\,.
\end{equation}
\end{proposition}

For the above bounds, the point is that the trivial bound on $\P_{\CN_{\R}(0,1)}(\CE)$, obtained by controlling the Gaussian measure by Lebesgue measure, is of order $N^{-m} \prod_{i=1}^m |J_i||J_i'|$ (this will be shown in the proof, but can also be understood on the heuristic level). For the error in \eqref{joint.compare} we save $\ll n^{-1/2+\eps}$ on this bound, while in \eqref{joint.ub} we obtain the same order upper bound for $\P(\CE)$ up to a tolerable loss of a factor $\lam^{-3m}\log^{O(m)}n$.

We commence with the proof of \Cref{prop:joint.gen}.
Let $K_*>0$ to be chosen sufficiently large and set $\delta=n^{-K_*}$. 
We first describe the event $\CE$ as a domain in $\R^{4\mom}$. 
Let $\cD$ denote the annulus 
\[
\cD:= B(0, C_0\sqrt{\log n})\setminus B(0, \log^{-\Cp/2}n)\subset\R^2.
\]
For $\Bb = (b,b')\in \R^2$ we write $\Bb^\perp := (b', -b)$, and define the rectangles
\begin{equation}	\label{def:Tb}
T_i(\Bb) = \bigg\{ \Ba \in \R^2 : \frac{\Ba\cdot \Bb^\perp}{\|\Bb\|_2} \in  \frac1n\cdot J_i\;, \;\; -\frac{ \Ba\cdot \Bb }{ \|\Bb\|_2^2 } \in  \frac{n}N\cdot J_i' \bigg\}\,,
\,\qquad 1\le i\le \mom,
\end{equation}
which have sides of length $ n\|\Bb\|_2 |J_i'|/N$ and $|J_i|/n$ in the direction of $\Bb$ and $\Bb^\perp$, respectively.
(Here we write $C\cdot J_i$ for the dilation of $J_i$ by a factor $C$.) 
Let
\begin{equation}	\label{def:U.Psi}
U_i = \Big\{ (a,a',b,b')= (\Ba,\Bb)\in \R^4: \Bb\in \cD, \;\Ba \in T_i(\Bb) \Big\}
\,,\qquad 
\CU= \prod_{i=1}^\mom U_i\,.
\end{equation}
Abbreviating henceforth
\begin{equation}	\label{def:tStSg}
\wt{S}:= \frac1{\sqrt{2n+1}}S_n(\bs t ), 	
\end{equation}
one sees that the left hand sides of \eqref{joint.compare} and \eqref{joint.ub} can be expressed as $|\P(\wt S\in \CU) - \P(\Gamma\in \CU)|$ and $\P(\wt S\in \CU)$, respectively. 

From the dimensions of the rectangles $T_i(\Bb)$ we have from Fubini's theorem that
\begin{equation}	\label{Ui.Leb}
\mLeb(U_i) 
= \int_\cD \mLeb(T_i(\Bb))d\Bb = \frac{\Delta}N |J_i||J'_i|
\end{equation}
where we denote
\begin{equation}	\label{CD.int}
\Delta:= \int_\cD \|\Bb\|d\Bb \sim \frac{2\pi C_0}3\log^{3/2}n.
\end{equation}
Thus,
\begin{equation}	\label{CU.Leb}
\mLeb(\CU) = (\Delta/N)^m \prod_{i=1}^m |J_i||J_i'| = \frac{\log^{O(m)}n}{N^m} \prod_{i=1}^m |J_i||J_i'| .
\end{equation}
For the measure of $\CU$ under the law of $\Gamma$, recall from \Cref{lem:cov} that the norm of the inverse of the covariance matrix of $\Gamma$ has operator norm of size $O(\lam^{-3m})$, and hence determinant of size $O(\lam^{-O(m^2)})$. 
By controlling the conditional density of $\Gamma$ in directions $(\Ba_1,\dots, \Ba_m)$ for fixed $(\Bb_1,\dots, \Bb_m)$ by the Lebesgue measure, and then integrating over $\cD^m$ under the marginal Gaussian measure, we get
\begin{equation}
\P(\Gamma\in \CU)  \ll_{\mom,\smooth, L_0}
\frac{1}{\lam^{O(\mom^2)} N^{\mom}  }\prod_{i=1}^\mom |J_i||J_i'|,
\end{equation}
giving \eqref{joint.ub.gaussian} as desired. 

We next note that the corners of the rectangles $T_i(\Bb)$ are $n^{O(L_0+1)}$-Lipschitz functions of $\Bb\in \cD$.
From this it follows that if $K_*$ is sufficiently large depending on $L_0$ and $m$, we can find sets $\CU_-\subset \CU\subset \CU_+$ such that $\CU_-$ and $\CU_+\setminus \CU_-$ are unions of cubes in $\R^{4m}$ of side length $\delta$ with disjoint interiors, and such that $\mLeb(\CU_+\setminus\CU_-) \le n^{-100}\mLeb(\CU)$ (say). 

The bound \eqref{joint.ub} now follows by
covering each cube in $\CU_+$ with balls of bounded overlap and applying the union bound, \Cref{thm:smallball}, and \eqref{CU.Leb}.

For \eqref{joint.compare}, we bound
\[
|\P(\wt S\in \CU) - \P(\Gamma\in\CU) | \le \P(\wt S\in \CU_+\setminus \CU_-) + \P(\Gamma\in \CU_+\setminus \CU_-)+ \sum_{Q} |\P(\wt S\in Q) - \P(\Gamma\in Q)| 
\]
where the sum runs over the cubes comprising $\CU_-$. Using the union bound and \Cref{thm:smallball} as we did for $\CU_+$, the first two terms above are of size 
\[
\ll \mLeb(\CU_+\setminus \CU_-) \ll n^{-100} \mLeb(\CU). 
\]
For the sum over $Q$, use \Cref{thm:box} to bound each term by $\ll_{m,\kappa, K_*} n^{-1/2} \mLeb(Q)$. Altogether we have
\[
|\P(\wt S\in \CU) - \P(\Gamma\in\CU) | \ll_{m,\kappa,K_*} n^{-1/2} \mLeb(\CU)
\]
and the claim now follows from \eqref{CU.Leb}.
This concludes the proof of \Cref{prop:joint.gen}.
\qed\\

\section{Proof of \Cref{prop:moments} for the real-valued case (moment comparison) }
\label{sec:proof-moments}

We condition on $\CG_{2}(\Cp/2)$ throughout the proof. As remarked before, in the real-valued case it suffices to work with $x_\al \in [0,\pi]$ because $P_n(-x) = \wb{P_n}(x)$. 
We allow implied constants to depend on $\mom$ and $\height$ without indication.
Recall also that $\smooth_0$ in the definition \eqref{def:Mn.sharp} of $\Points^\sharp_n$ is an absolute constant. 
For $\bs\al = (\al_1,\dots, \al_m)\in [N]^m$ we denote events
\[
\CE(\bs\al) := \bigg\{\bigwedge_{i\in [\mom]}  |X_\ali| \le \height \bigg\}\,.
\]
We have
\begin{align}
\E \Big( \Points^\sharp_n\big([-\height,\height]\big)^\mom\Big)  &= \sum_{\bs\al\in E} \pr(\CE(\bs\al))  
=  \sum_{ \bs \al \in E'}  \pr(\CE(\bs\al))  
  +  \sum_{\bs\al \in E\setminus E'}   \pr(\CE(\bs\al))  	\label{mom-split}
\end{align}
where
\begin{align*}
E &:= \big\{\bs\al= (\al_1,\dots, \al_\mom)\in [N/2]^\mom : x_{\al_1},\dots, x_{\al_\mom} \notin \badarcs(\smooth_0)  \big\} \,,\\
E' &:= \big\{  \bs \al \in E:  |x_{\al_i}-x_{\al_j}|> 4\pi/n\; \forall 1\le i<j\le \mom \big\}\,.
\end{align*}

Note that if $x,x' \in [0,\pi]$ such that $|\frac{x-x'}{2\pi}|\ge \frac{\la}{n}$ then we also have $\frac{\la}{n} \le \frac{x+x'}{2\pi} \le 1- \frac{\la}{n}$.
Hence within $E'$ the angles are $1$-spread and by Proposition \ref{prop:joint} 
\[
\bigg| \sum_{\bs\al\in E'} \pr\big(\CE(\bs\al)\big)
- \sum_{\bs\al\in E'} \pr_\Bg \big(\CE(\bs\al)\big)
\bigg| 
\le N^{\mom} 
o(N^{-m}) = o(1).
\]
It only remains to bound the sum over $\bs\al\in E\setminus E'$.

By \Cref{lem:neararcs}, under $\CG_{2}(\Cp/2)$, it suffices to consider 
$m$-tuples of the form
\begin{equation}	\label{form.pairs}
(\al_1,\dots, \al_{\mom -k}, \al_1+1, \al_2+1,\dots, \al_k+1)
\end{equation}
consisting of $k$ pairs of points $(\al_l,\al_l+1)$ that are immediate neighbors, for some $0\le k\le \mom/2$, while the $\mom-k$ points
$x_{\al_1},\dots, x_{\al_{\mom-k}}$ are separated by at least $4\pi/(n \log^{3\Cp}n)$ in $[0,\pi]$. Note also that by the remark above we also have $\frac{x_{\al_i}+ x_{\al_j}}{2\pi} \ge 1/(n \log^{3\Cp}n)$. 


We divide this class of such $\bs\al$ into two sets $E_1, E_2$, where
$E_1$ is the set of $\bs\al\in E\setminus E'$ of the form \eqref{form.pairs} (possibly with $k=0$) such that $|x_{\al_i}-x_{\al_j}|\le 4\pi/n$ for some $1\le i<j\le \mom-k$,
and $E_2$ is the set of $\bs\al\in E\setminus E'$ of the form \eqref{form.pairs} with $k\ge 1$ and 
$|x_{\al_i}-x_{\al_j}|> 4\pi/n$ for all $1\le i<j\le \mom-k$.

For the sum over $E_1$, 
we have $|E_1|=O(N^{\mom -k}/n)$ since there are $O(N/n)$ options for the close point with all others fixed. 
As the points $x_{\al_1},\dots, \dots,  x_{\al_{\mom -k}}$ are separated by at least $4\pi/(n \log^{3\Cp}n)$, 
from the upper bound \eqref{joint.ub} in \Cref{prop:joint.gen} with $J_i \equiv [-\height,\height]$ and $J_i'\equiv [-\pi,\pi]$, 
we have
\begin{align*}
\sum_{\bs\al\in E_1} \pr\big( \CE(\bs\al)\big)
\ll (N^{\mom -k}/n)\times  N^{-\mom} \height^\mom \log^{O(K_0m)}n
\ll \frac1n\log^{O(m)}n = o(1).
\end{align*}

For the sum over $E_2$, 
by \Cref{lem:neararcs}, under $\CG_{2}(\Cp/2)$ we have the containment of events
\[
\big\{ |X_{\al_i}| \le \height \;,\; |X_{\al_i+1}| \le \height \big\} 
\subset 
\bigg\{  |X_{\al_i}| \le \height \;,\; Y_{\al_i}  \in \Big[ \frac\pi{N} - \frac{\pi}{N \log^{\Cp/4}n}  , \frac{\pi}N \Big]\bigg\}\,,
\]
so for each such $\bs\al$ we can bound
\[
\pr\big(\CE(\bs\al)\big)
\le
\pr\bigg( \, Y_{\al_1}  \in \Big[ \frac\pi{N} - \frac{\pi}{N \log^{\Cp/4}n}  , \frac{\pi}N \Big]\,,\,\bigwedge_{i\in [\mom-k]}  |X_\ali| \le \height \, \bigg) .
\]
Applying \eqref{joint.compare} with $\mom-k$ in place of $\mom$, $\lam = 1/2$ (say), $J_i\equiv[-\height,\height]$, $J_1' = [\pi( 1- \log^{-\Cp/4}n), \pi]$, and $J_i'=[-\pi,\pi]$ for $2\le i\le \mom-k$, the right hand side above is bounded by
\[
\pr_{\CN_\R(0,1)}\bigg( \, Y_{\al_1}  \in \Big[ \frac\pi{N} - \frac{\pi}{N \log^{\Cp/4}n}  , \frac{\pi}N \Big]\,,\,\bigwedge_{i\in [\mom-k]}  |X_\ali| \le \height \, \bigg) +o(N^{-(m-k)}).
\]
Finally, we apply \eqref{joint.ub.gaussian} to bound the first term above by $o(N^{-(m-k)})$.
Combining the preceding displays and summing over $\bs\al\in E_2$ gives
\begin{align*}
\sum_{\bs\al\in E_2} \pr\big( \CE(\bs\al)\big) =o(1).
\end{align*}
We have thus shown that the sum over $\bs\al\in E\setminus E'$ in \eqref{mom-split} is $o(1)$, which completes the proof of \Cref{prop:moments}.\qed

\section{Proof of \Cref{thm:main} (main result)}
\label{sec:proof-main}

We fix $\kappa=\kappa_0$ as in \Cref{lem:badarcs}, and let $\tau>0$ be arbitrary. As in the previous section we allow implied constants to depend on $\mom$ and $\tau$.
It follows from \Cref{prop:moments} that
\[
\lim_{n\to \infty} \Big|\P\big({\Points^\sharp_{\Bg}}([-\height,\height])=0\big) - \P\big(\Points^\sharp([-\height,\height])=0\big)\Big| =0.
\]
On the other hand, by \Cref{thm:YZ} and \Cref{lem:badarcs},
\[
\lim_{n\to \infty} \bigg|\P_{\Bg}\Big(m_n > \frac{\height}{n}\Big) - \P_{\Bg}\Big(\Points^\sharp([-\height,\height])=0\Big) \bigg|=0
\]
and hence it suffices to show
\[
\lim_{n\to \infty} \bigg|\P\Big(m_n > \frac{\height}{n}\Big) - \P\Big(\Points^\sharp([-\height,\height])=0\Big) \bigg|=0.
\]
To this end, recall that on the event $\CG_{2}(\Cp)$,
\begin{equation}\label{eqn:PF}
|P(x) - F_\al(x)| \le N^{-2} \sup_{x\in [-\pi,\pi]} |P''(x)|  \le \frac{\log^{3\Cp} n}{n^2}
\end{equation}
for all $x\in I_\al$.
By Lemma \ref{lem:badarcs} we have
\begin{align*}
&
\bigg|\P\Big(m_n > \frac{\height}{n}\Big) - \P\Big(\Points^\sharp([-\height,\height])=0\Big) \bigg| \\
&\qquad\qquad \le \Pr\Big( m_n > \frac{\height}{n}\;,\; \Points^\sharp([-\height,\height])\ge1 \Big)
+ \Pr\Big( m_n \le \frac{\height}{n}\;,\; \Points^\sharp([-\height,\height])= 0 \Big)\\
&\qquad\qquad\le \sum_{\al\in[N]: x_\al \notin \badarcs(\smooth)} \P\Big(\CG_{2}(\Cp) \,\wedge\, |X_\al| < \height \,\wedge\, \min_{x\in I_\al} |P(x)| \ge \height/n \Big)\\
&\qquad\qquad\qquad\qquad+
\sum_{\al\in[N]: x_\al \notin \badarcs(\smooth)} \P\Big(\CG_{2}(\Cp) \,\wedge\, |X_\al| \ge \height \,\wedge\, \min_{x\in I_\al} |P(x)| < \height/n \Big) +o(1)\\
&\qquad\qquad \le \sum_{\al\in[N]: x_\al \notin \badarcs(\smooth)} \P \bigg(|X_\al|\in  \bigg[\height - \frac{\log^{3\Cp} n}{n} ,  \height + \frac{\log^{3\Cp} n}{n} \bigg]\bigg)+o(1),
\end{align*} 
where we used the definition of $X_\al$ and \eqref{eqn:PF} in the last estimate.

Applying the bound \eqref{joint.ub} of \Cref{prop:joint.gen} with $\mom=1$, $J_1=[\height-n^{-1}\log^{3\Cp}n,\height+n^{-1}\log^{3\Cp}n]$, $J_1'=[-\pi,\pi]$, and $\lam=1$, say (with a single point $x_\al$ being trivially $\lam$-spread), we have 
\[
\P \bigg(X_\al\in  \bigg[\height - \frac{\log^{3\Cp} n}{n} ,  \height + \frac{\log^{3\Cp} n}{n} \bigg]\bigg) 
\ll\frac{\log^{3\Cp}n}{nN}
\]
for each $\al$ with $x_\al\notin \badarcs(\smooth)$, as well as the same bound for the event with $X_\al$ replaced by $-X_\al$. 
From the union bound and summing over $\al$ we conclude
\[
\bigg|\P\Big(m_n > \frac{\height}{n}\Big) - \P\Big(\Points^\sharp([-\height,\height])=0\Big) \bigg| 
\ll \frac{\log^{3\Cp}n}{n} +o(1) = o(1)
\]
as desired.

\section{Proof of  \Cref{thm:smallball}}
\label{sec:smallball}
Fix $\bs t$ as in the theorem statement. 
Recall the notation $S_n=S_n(\bs t)$ (we henceforth suppress $\bs t$) and $\bs w_j$ from \eqref{def:Snt} and \eqref{def:wj}.
Let $t_0=\delta^{-1}$  and  let $\phi_j$ denote the characteristic function of $\xi_j\bs w_j$. By a standard 
\revised{consequence of Esseen's inequality (see e.g.\ \cite[Lemma 7.17]{TaVu:book} and its proof)}
we can bound the small ball probability by 
$$\P(\frac{1}{\sqrt{2n+1}} S_n \in B(w,\delta) ) \le C_\mom   (\frac{n}{t_0^2})^{4\mom /2} \int_{\R^{4\mom }} 
\revised{\bigg|}\prod_{j=-n}^n \phi_j(u) 
\revised{\bigg|}
e^{-\frac{n \|u\|_2^2}{2 t_0^2}} du =: J_1+J_2+J_3,$$
where in $J_1, J_2, J_3$ the integral is restricted to the ranges  $\|u\|_2 \le r_0 =O(1)$, $r_0 \le \|u\|_2 \le R=n^{K_*}$, and $\|u\|_2>R$, respectively for $K_*>0$ to be chosen sufficiently large.

For $J_1$, from \eqref{Phi.gauss} and \eqref{gauss.LB} below we can bound
\[
\bigg|\prod_{j=-n}^n \phi_j( u)\bigg| \le \exp\bigg(-c \inf_{a_1\le |a| \le a_2} \sum_j \| a \langle \wvec_j, u/2\pi \rangle\|_{\R/\Z}^2 \bigg).
\]
Thus, if $r_0$ is sufficiently small, then we have $\|a \langle \wvec_j, u / 2\pi \rangle\|_{\R/\Z} = |a|\|\langle \wvec_j, u/ 2\pi  \rangle\|_2$, and so from \Cref{lem:cov} we have 
$$\sum_j \| a \langle \wvec_j, u/2\pi \rangle\|_{\R/\Z}^2
\ge c' n\|u\|_2^2\min(\lam,1)^{6\mom-3}.$$
Hence
\begin{align*}
J_1 &=  C_\mom  (\frac{n}{t_0^2})^{2\mom} \int_{\|u\|_2 \le r_0} \prod_j \phi_j(u) e^{-\frac{n \|u\|_2^2}{2 t_0^2}} du \\
& \le  C_\mom  (\frac{n}{t_0^2})^{2\mom} \int_{\|u\|_2 \le r_0}  e^{-\frac{n \|u\|_2^2}{2 t_0^2} - c'n \|u\|_2^2\lam^{6\mom-3}} du\\
& = O_\mom(\frac{1}{\lam^{3\mom}( t_0^2 +1)^{2\mom}})= O_\mom(\lam^{-3\mom}\delta^{4\mom }).
\end{align*}

For $J_2$, recall by Theorem \ref{thm:char} that for $r_0 \le \|u\|_2 \le R=n^{K_*}$ we have
$$ |\prod_{j=-n}^n  \phi_j( u)| =O(e^{-\log^2 n}).$$
Thus 
\begin{align*}
J_2 &=  C_\mom  (\frac{n}{t_0^2})^{2\mom} \int_{r_0 \le \|u\|_2 \le R} \prod_{j=-n}^n \phi_j(u) e^{-\frac{n \|u\|_2^2}{2 t_0^2}} du \\
& \le  C_\mom   (\frac{n}{t_0^2})^{2\mom} \int_{r_0 \le \|u\|_2 \le R}  e^{-\log^2 n}  e^{-\frac{n \|u\|_2^2}{2 t_0^2}}du\\
&\ll_\mom n^{O_{\mom,K_*}(1)} e^{-\log^2 n}  \ll_{\mom,K} e^{-\frac12\log^2 n} .
\end{align*}

For $J_3$, we have
\begin{align*}
J_3 &=  C_\mom (\frac{n}{t_0^2})^{4\mom/2} \int_{ \|u\|_2 \ge n^{K_*}} \prod_{j=-n}^n \phi_j(u) e^{-\frac{n \|u\|_2^2}{2 t_0^2}} du = O_\mom(e^{-n})
\end{align*}
for $K_*$ sufficiently large.

\section{Proof of \Cref{thm:box}}
\label{sec:box}

For the proof we make use of a quantitative Edgeworth expansion for the distribution of $S_n=S_n(\bs t)$ (we will suppress the dependence of $S_n$ on $\bs t$ in much of what follows).
Our treatment is similar to \cite{DNN}. 
Let
\begin{equation}\label{eqn:V_n}
V_n :=\frac{1}{2n+1} \sum_{j=-n}^n \bs w_j \bs w_j^\tran.
\end{equation}
be the covariance matrix of $S_n/\sqrt{2n+1}$.
Let $\widetilde Q_n$ denote the distribution of $S_n/\sqrt{2n+1} $, and  let $\widetilde Q_n(x)$ denote the cumulative distribution function for this distribution. 
The theorem below shows that $\widetilde Q_n$ is asymptotically $\widetilde{Q}_{n,\infty}$, where
\begin{equation}\label{e.Q_nl}
\widetilde{Q}_{n,\ell} := \sum_{r=0}^{\ell-2} n^{-r/2} T_r(-\Phi_{0,V_n}, \{\overline{\chi}_\nu\}),\qquad \ell\ge 2,
\end{equation}
for (signed) measures $T_r(-\Phi_{0,V_n}, \{\overline{\chi}_\nu\})$ to be defined below. For convenience, the density of $\widetilde Q_{n,\ell}$ is denoted by $Q_{n,\ell}$ while the density of $\widetilde Q_n$ is denoted by $Q_n$.

Let $W$ be the standard Gaussian vector in $\R^{4\mom}$. For any covariance matrix $V$,  $V^{1/2}W$ is the Gaussian random vector in $\R^{4\mom}$ with mean zero and covariance $V$. Let $\phi_{0,V}$ denote the density of its distribution and let $\Phi_{0,V}$ denote the cumulative distribution function. If $V$ is the identity matrix then we simply write $\phi$ and $\Phi$, respectively.  Recall that the cumulants of a random vector $X$ in $\R^{4\mom}$ are the coefficients in the following formal power series expansion 
\begin{equation}\label{e.cumulant}
\log \E [ e^{z \cdot X}] = \sum_{\nu\in \BN^d} \frac{\chi_{\nu} z^\nu}{|\nu|!} 
\,,\qquad  z\in \C^{4\mom}.
\end{equation}
From the independence of the random coefficients $\xi_j$, it follows that the cumulants of $S_n$ are the sum of the corresponding cumulants of $\xi_j\bs w_j$, which in turn are polynomials in the moments of $\xi$ and the entries of $\bs w_j$.  Let $\overline{\chi}_\nu := \chi_{\nu}(S_n)/(2n+1)$, which is the average of cumulants of $\xi_j\bs w_j,  -n\le j\le n$.  

Note that cumulants  of $V_n^{1/2}W$ match the cumulants of $S_n/\sqrt{2n+1}$ for any $|\nu|\le 2$, while the higher order cumulants of the Gaussian vector $V_n^{1/2}W$ vanish. Therefore,
\begin{eqnarray*}
\log \E [ e^{z \cdot  (S_n/\sqrt{2n+1} )}]  
&=&   \log  \E [e^{z\cdot (V_n^{1/2}W)}]  +   \sum_{\nu\in \BN^d: |\nu|\ge 3}  (n\overline{\chi}_{\nu})  \frac{z^\nu}{|\nu|!}  n^{-|\nu|/2} \\ 
&=& \log  \E [e^{z\cdot V_n^{1/2}W}] +    \sum_{\ell\ge 1} (\sum_{\nu\in \BN^d: |\nu|=\ell+2}  \overline{\chi}_{\nu} \frac{z^\nu}{|\nu|!})  n^{-\ell/2}.
\end{eqnarray*}
Letting $\overline\chi_\ell(z) = \ell!\sum_{\nu\in \BN^{4\mom}: |\nu|=\ell}  \overline{\chi}_{\nu} z^\nu / |\nu|!$ for all $z\in \C^{4\mom}$, we obtain
\begin{eqnarray*}
\E [ e^{z \cdot (S_n/\sqrt{2n+1} )}]/ \E [ e^{z \cdot   V_n^{1/2}W}] 
&=&  \exp[\sum_{\ell\ge 1}  \frac{\overline{\chi}_{\ell+2}(z)}{(\ell+2)!}   n^{-\ell/2}] \\
&=&  \sum_{m\ge 0} \frac 1{m!} \Big(\sum_{\ell\ge 1}  \frac{\overline{\chi}_{\ell+2}(z)}{(\ell+2)!}   n^{-\ell/2}\Big)^m\\
 &= & \sum_{\ell\ge 0}  \wt T_\ell n^{-\ell/2},
\end{eqnarray*}
where $\widetilde T_\ell$ is obtained by grouping terms of the same order $n^{-\ell/2}$. 
It is clear  that $\widetilde T_{\ell}$ depends only on $z$ and the average cumulants $\overline{\chi}_{\nu}, |\nu|\le \ell+2$. We will write $\widetilde T_\ell(z, \{\overline{\chi}_\nu\})$ to stress this dependence.  Replacing $z$ by $iz$, we obtain the following expansion for the characteristic function of $S_n/\sqrt{2n+1}$:
\begin{eqnarray*}
\E [ e^{iz \cdot (S_n/\sqrt{2n+1} )}] 
&=&   \E [ e^{iz \cdot   V_n^{1/2}W}] \sum_{\ell\ge 0}  \widetilde T_\ell (iz, \{\overline{\chi}_\nu\})n^{-\ell/2}.
\end{eqnarray*}
Next, let 
$D=(D_1,\dots, D_{4\mom})$ 
be the partial derivative operator and let $\widetilde{T}_{\ell}(-D, \{\overline{\chi}_\nu\})$ be the differential operator obtained by formally replacing all occurences of $iz$ by $-D$ inside $\widetilde T_\ell (iz, \{\overline{\chi}_\nu\})$.  
We define the signed measures $T_{\ell}(-\Phi_{0,V_n}, \{\overline{\chi}_\nu\})$ in \eqref{e.Q_nl} to have the following density with respect to the Lebesgue measure:
$$T_{\ell}(-\phi_{0,V_n},  \{\overline{\chi}_\nu\})(x):=  \Big(\widetilde{T}_{\ell}(-D, \{\overline{\chi}_\nu \})\phi_{0,V_n}\Big)(x).$$

The following result gives a quantitative comparison between $\widetilde Q_n$ and $\widetilde{Q}_{n,\ell}$; cf.\ also \cite[Theorem 4.1]{DNN}. 
For convenience of notation, for each $\ell>0$, let 
$$\rho_\ell := \frac{1}{n} \sum_{-n\le j\le n}  \|\bs w_j\|_2^\ell\cdot \E |\xi|^\ell.$$ 
Thus 
$\rho_\ell=O_{\ell,\mom}(\E |\xxi|^\ell)=O_{\ell,\mom}(1)$ if $\xi$ is sub-Gaussian. To stay slightly more general, here we only assume that $\xxi$ has bounded moments up to some sufficiently large order.
For a given measurable function $f:\R^{4\mom}\to \R$, define 
$$M_\ell(f) := \sup_{x\in \R^{4\mom}} \frac{|f(x)|}{1+\|x\|_2^\ell}.$$

\begin{theorem}[Edgeworth expansion]
\label{thm:EW}
Assume $\E |\xxi|^{\ell+4\mom+1} <\infty$ for some  $\ell \ge 4$. Let $f:\R^{4\mom}\to \R$ be a measurable function such that $M_{\ell}(f)<\infty$.  Suppose that $\bs t=(t_1,\dots, t_\mom)$ is $n^\smooth$-smooth and $1$-spread for some $\smooth>0$. Then for any fixed $K_*>0$ and any $n^{-K_*}\le \eps\le 1$,
\begin{eqnarray*}
&& |\int f(x) d\widetilde Q_{n}(x) - \int f(x) d \widetilde{Q}_{n,\ell}(x) |\\
 &\le& C M_{\ell}(f) (n^{-(\ell-1)/2} + e^{-\log^2 n}  )+  \overline{\omega}_f(2\eps: \sum_{r=0}^{\ell+4\mom-2} n^{-r/2} T_r(-\phi_{0,V_n}: \{\overline{\chi}_\nu\})
 \end{eqnarray*}
where  for a density $\phi$,
$$\overline{\omega}_f(\eps:\phi) = \int (\sup_{y\in B(x,\eps)} f(y) -\inf_{y\in B(x,\eps)} f(y)) d\phi(x),$$ 
for some $C=C(\{\rho_k, k\le \ell\}, \smooth,K_*)>0$.
\end{theorem}
 

\begin{proof}[Proof of Theorem \ref{thm:EW}] This follows from \cite[Section 4]{DNN} (which in turns follows the approach of \cite{BhRa:book3} with some important modifications, see also \cite{BCP:variance}). For completeness we sketch the proof below. Let $d=4m$. For convenience, we assume that $\eps =n^{-K_*}$ and denote
$$\widetilde H_n = \widetilde Q_n - \widetilde Q_{n,\ell},$$
and let $H_n$ be its density. As usual the characteristic function of $H_n$ is   $\widehat {H_n}(\eta) =  \int_{\R^d} e^{i t \cdot \eta} \widetilde H_n(dt)$.

Let $\widetilde K$ be a probability measure supported inside the unit ball $B(0,1)=\{x\in \R^d: \|x\| \le 1\}$ (whose density is denoted by $K$) such that its characteristic function $\widehat K(\eta)$ satisfies
\begin{equation}\label{e.Keps}
|D^\alpha \widehat K(\eta)|  = O( e^{-\|\eta\|_2^{1/2}}), \quad |\alpha| \le \ell+d+1.
\end{equation}
Such a measure could be constructed using elementary arguments, see for instance \cite[Section 10]{BhRa:book3}. We then let $\widetilde K_\epsilon$ be the $\epsilon$-dilation of $K$, namely $\widetilde K_\epsilon(A) = \widetilde K(\epsilon^{-1}A)$ and $\epsilon^{-1}A := \{x/\epsilon: x\in A\}$ 
for all measurable $A$. Some simple computation yields
\begin{eqnarray*}
|\int f(y) d\widetilde H_n(y)| & \le & C_\ell M_{\ell}(f) \int (1+\|t\|_2)^{\ell} |H_n*K_\epsilon|(t)dt  +  \bar{\omega}_f(2\eps: |\widetilde Q_{n,\ell}|)\\
&= & O \Big(\max \{ \int |D^\alpha (\widehat {H_n})(\eta) D^{\beta} (\widehat {K_\epsilon})(\eta)| d\eta: \ \ |\alpha|+|\beta|\le \ell + d + 1\}\Big).
\end{eqnarray*}
Following \cite{BhRa:book3} (see \cite[Corollary 4.3]{DNN} for a different proof) we can show that for some $c_1>0$ sufficiently small we have
\begin{eqnarray*}
\int_{\|\eta\|_2\le c_1 \sqrt n}|D^\alpha \widehat {H_n}(\eta) D^{\beta} \widehat {K_\epsilon}(\eta)| d\eta 
&=& O\Big(\int_{\|\eta\|_2\le c_1 \sqrt n}|D^\alpha \widehat {H_n}(\eta) | d\eta\Big)\\
&=&  O ( n^{-(\ell+d-1)/2}).
\end{eqnarray*}
It thus remains to consider the range $\|\eta\|_2\ge c_1 \sqrt n$. We use triangle inequality to estimate (where $Q_n$ is the density of $\wt{Q}_n$)
\begin{align*}
\int_{\|\eta\|_2\ge c_1 \sqrt n} |D^{\alpha} \wh{H}_n(t) D^{\beta} \wh{K}_\eps |d \eta  &\le \int_{\|\eta\|_2 \ge c_1 \sqrt n} |D^{\alpha} \wh{Q}_n(t) D^{\beta} \wh{K}_\eps |d\eta \\
&+\int_{\|\eta\|_2 \ge c_1\sqrt  n} |D^{\alpha} ( \sum_{r=0}^{\ell-2+d} n^{-r/2} P_r(i\eta: \{\chi_{\nu,n}\})) \exp(-1/2\langle \eta,B_n \eta \rangle) |d \eta,
\end{align*}
where $B_n^2=V_n^{-1}$ (defined in \eqref{eqn:V_n}.)

The second term can be controlled by $O(e^{-cn})$  thanks to the Gaussian decay of $\exp(-1/2\langle \eta,B_n \eta \rangle)$.

Let $\phi_i (\eta) = \E e^{ i \eta \cdot \bw_i}$. Then for $|\alpha|\le \ell+d+1$ we have  $D^\alpha_{\eta}(\phi_i(\eta/\sqrt n)) = n^{-|\alpha|/2} O(\E\|X_{n,i}\|_2^{|\alpha|}) = O(1)$. Thus,
$$ |D^{\al} \widehat{Q}_n(\eta)|  = |D^{\alpha}(\prod_{i=1}^n \phi_i(\frac{\eta}{\sqrt n}))| = O(\sum_{\gamma_1+\dots+\gamma_n  = \al }   |\prod_{i=1, \gamma_i =0}^n \phi_i(\frac{\eta}{\sqrt{n}})|),$$
while we also have $|D^\beta \wh{K}_{\eps}(\eta)| = O(\eps^{|\beta|} e^{-(\eps \|\eta\|_2)^{1/2}}) = O(e^{-(\eps \|\eta\|_2)^{1/2}})$. Thus,  it remains to control, for each $(\gamma_1,\dots,\gamma_n)$ with $|\gamma_1|+\dots +|\gamma_n| \le \ell+d+1$ and each $r>0$ independent of $n$:
\begin{align*}
J_\gamma(n,\eps) &= \int_{\|\eta\|_2 \ge r\sqrt n}  |\prod_{i=1, \gamma_i =0}^n \phi_i(\frac{\eta}{\sqrt n})|  e^{-(\eps  \|\eta\|_2)^{1/2}}  d\eta\\
&= n^{d/2} \int_{\|\eta\|_2\ge r} |\prod_{i=1, \gamma_i =0}^n \phi_i(\eta)| e^{-(n^{-K_\ast +1/2}\|\eta\|_2)^{1/2}} d\eta.
\end{align*}

Clearly it suffices to consider $r \le \|\eta\|_2 \le n^{K_\ast-1/2 + \tau}$ because the integral for $\|\eta\|_2 \ge n^{K_\ast-1/2 + \tau}$ is extremely small. Again, because $\al$ is fixed, by throwing away from the set $\{\Bw_i\}$ a fixed number of elements, let us assume that $\al=0$ for simplicity \footnote{In the general case $\al \neq 0$ we apply Theorem \ref{thm:char'} instead of Theorem \ref{thm:char}.}. To this end, by Theorem \ref{thm:char} for sufficiently large $n$ we have
 $$|\prod_i \phi_i(  \eta)| \le e^{-\log^2n}.$$
 Thus we just shown that, with $\eps= n^{-K_\ast}$ we have $J_\gamma(n,\eps)= O(e^{- \log^2 n})$, completing the proof.
\end{proof}

We turn now to the proof of \Cref{thm:box}.
We follow \cite[Section 5]{DNN} with some slight modifications. 
\revised{We are free to assume $K$ is larger than any fixed constant. By approximating $Q$ with a union of smaller boxes with disjoint interiors it suffices to establish the claim for boxes of the form
$Q=w+ B_m(\bs \delta)$ with $B_m(\bs \delta):= \prod_{i=1}^4[-\delta_i,\delta_i]^m\subset \R^{4m}$
for arbitrary $\delta_i\in [n^{-2K},1]$ for $1\le i\le 4$ (assuming $K\ge1$, say).
}
Let $\eta, \eps>0$ to be chosen later, and towards an application of \Cref{thm:EW} we fix some 
\revised{$K_*>2K$.}
In the sequel we abbreviate 
\revised{$\delta:= n^{-2K}$.}
We let 
\[
g:= \frac{1}{16\delta_1\delta_2\delta_3\delta_4} 1_{w+B_\mom(\bs \delta)}
\]
be the $L^1$-normalized indicator for the box $w+B_\mom(\bs \delta)\subset \R^{4\mom}$.
For $1\le i\le 4$ let $\varphi_{i,\eta}: \R\to [0,1]$ be a $C^\infty(\R)$ function with support inside $[-\delta_i, \delta_i]$ such that 

(i)  $\varphi_{i,\eta}(x) = \delta_i^{-1}$ for $|x|\le \delta_i(1-\eta)$, and

(ii) $|\varphi_{i,\eta}^{(k)} (x)|  = O_k(\delta_i^{-(k+1)}\eta^{-k})$ for any $k\ge 0$,\\
and set
\[
f(\bs x) = \prod_{r=1}^\mom \prod_{i=1}^4 \varphi_{i,\eta}(w^i_r + x^i_r)
\]
where we write $\bs w=(w^1,\dots,w^4), \bs x=(x^1,\dots, x^4)\in \R^{4\mom}$.
We have
 $$\|\nabla f(\bs x)\|_2  \ll_\mom  \frac 2 {\delta^{4\mom +1} \eta}$$
uniformly in $\bs x$.
Recall that $\bar{\omega}_f(\eps:\phi) = \int (\sup_{y\in B(x,\eps)} f(y) -\inf_{y\in B(x,\eps)} f(y)) \phi(x)dx$, and $\phi$ is the density of a Gaussian vector.   Consequently, for any polynomial $p(x)$ with bounded degree and bounded coefficients we have
\begin{equation*} 
\bar{\omega}_{f} (\eps:  p(x)\phi_{0,V_n}(x)) = O (\eta^{-1}\delta^{-4\mom-1}  \eps),
\end{equation*}
where the implied constant depends on the eigenvalues of $V_n$, and on the degree and coefficients of $p$. 
In particular, the final error term in Theorem \ref{thm:EW} can be expressed as
$$\sum_{r=0}^{\ell+4\mom-2} n^{-r/2} T_r(-\phi_{0,V_n}: \{\overline{\chi}_\nu\})=p(x)\phi_{0,V_n}(x)$$
for some polynomial $p$ with degree at most $4\mom+\ell$ and coefficients bounded by the first $4\mom+\ell$ moments of $\xi$. Therefore
\begin{equation}\label{e.modcont}
\bar{\omega}_{f} (2\eps:  \sum_{r=0}^{\ell+4\mom-2} n^{-r/2} T_r(-\phi_{0,V_n}: \{\overline{\chi}_\nu\})) = O (\eta^{-1}\delta^{-4\mom-1}  \eps), 
\end{equation}
where the implied constant depends on the eigenvalues of $V_n$ and the moments up to order $O(\mom)$ of $\xxi$.

Recall the shorthand notation $\wt S:= S_n(\bs t)/\sqrt{2n+1}$ from \eqref{def:tStSg}, and that $\Gamma$ has the distribution of $\wt S$ with standard real Gaussians in place of the variables $\xi_j$. 
From Theorem~\ref{thm:smallball} and Corollary \ref{cor:smallball}, 
\begin{eqnarray*}
 \big|  \E f(\wt S) - \E g(\wt S) \big|    
&\le& (\frac{C}{\delta})^{4\mom} \sum_{r=1}^\mom  \P(||\Re P_{n}(s_r)|-\delta_1|\le \eta \delta_1)+\P(|\Re P_{n}'(s_r)|-\delta_2|\le \eta \delta_2)  \\
& +&  \P(|| \Im P_{n}(s_r)|-\delta_3|\le \eta \delta_3)+\P(|\Im P_{n}'(s_r)|-\delta_4|\le \eta \delta_4)\Big) \\
&\ll_\mom &(\frac{C}{\delta})^{4\mom} \eta  
\end{eqnarray*}
\revised{where we used in the last line that $\delta_i\le1$ for each $1\le i\le 4$.}
\revised{Recalling the notation $M_\ell(f)$ from \Cref{thm:EW}, we have $M_\ell(f) \le \|f\|_\infty = O(1/\delta)^{4\mom}$ for any $\ell\ge0$.}
 By Theorem \ref{thm:EW} and \eqref{e.modcont} (with $\ell =16\mom K+3$), after keeping the first term of the expansion, and by the triangle inequality we have
\begin{align*}
\Big |\E f(\wt S)  -  \E f( \Gamma)\Big | 
 &\le  \Big |\int f  (x) \sum_{r=1}^{\ell-2} n^{-r/2} T_r(-\phi_{0,V_n}(x), \{\overline{\chi}_\nu\}) \Big | \\
 & + M_{\ell}(f) O\Big(n^{- 8\mom K-1} + e^{-\log^2 n}  \Big) + \bar{\omega}_{f} (2\eps:  \sum_{r=0}^{4\mom +1} n^{-r/2} T_r(-\phi_{0,V_n}: \{\overline{\chi}_\nu\})) \\
 &=    O(n^{-1/2}) + (C/\delta)^{4\mom } O (n^{-8\mom K-1} + e^{-\log^2 n}  ) + O ((C/\delta)^{4\mom +1}\eta^{-1} \eps ),
 \end{align*}
 where we used the fact that $|\int f(x) T_r(-\phi_{0,V_n}(x), \{\overline{\chi}_\nu\})|=O(1)$ with the implied constant depending on the moments of $\xxi$ up to order $r$ and on the implicit constant from (ii) of $\varphi$.
In particular, the above is also true for the Gaussian case. Consequently, again by the triangle inequality
\begin{eqnarray*}
  |\E  g(\wt S)  -  \E g(\Gamma) | 
&\le&    |\E  g(\wt S) - \E f(\wt S)|  
+ |\E f(\Gamma) - \E g(\Gamma)| 
+  |\E f  (\wt S) - \E f  (\Gamma) |    \\
&\ll_m&  n^{-1/2}  + (\frac{1}{\delta})^{4\mom } (n^{-8\mom K-1} + e^{-\log^2 n}  + \delta^{-1}\eta^{-1}\eps+ \eta 
)  = O(n^{-1/2}),
\end{eqnarray*}
where we took $\eta=\eps^{1/2}$ and $\eps=n^{-K_*}$ with $K_*$ sufficiently large compared to $K$.

\section{Proof of \Cref{thm:char}}
\label{sec:char}

We assume throughout this section that $n$ is sufficiently large depending on $\mom,\smooth,K_*$ and the sub-Gaussian constant for $\xi$.
We first recall a definition and fact from \cite{TaVu:circ}.
For a real number $w$ and a random variable $\xi$, define the $\xi$-norm of $w$ as
\[
\|w\|_\xi := ( \E \|w(\xi - \xi')\|_{\R/\Z}^2)^{1/2},
\]
where $\xi'$ is an iid copy of $\xi$. 
For instance, if $\xi$ has the Rademacher distribution $\pr(\xi=\pm1) = 1/2$, then $\|w\|_\xi^2 = \|2 w\|_{\R/\Z}^2/2$.
For any real number $w$ we have
\[
|\E e(w \xi)| \le \exp(-c\|w/2\pi\|_\xi^2)
\]
for an absolute constant $c>0$.

Now with $\charr_j: \R^{4\mom}\to \C$ the characteristic function of $\xi_j \bs w_j$, we have
\begin{align}	\label{Phi.gauss}
\Big|\E e\big(  \langle S_n(\bs t), \bs x\rangle\big)\Big|
&= |\prod_j \charr_j(\bs x)|
= \prod_j | \E e(\xi_j \langle \bs w_j, \bs x\rangle)|
\le \exp ( - c\sum_j \|\langle \bs w_j , \bs x/2\pi\rangle \|_\xi^2).
\end{align}
Furthermore, as $\xi$ is sub-Gaussian and of unit variance, there exist positive constants $a_1,a_2, c>0$ depending only on the sub-Gaussian moment of $\xi$ such that $\pr(a_1<|\xi-\xi'|<a_2) \ge c$, and so
\begin{align}	\label{gauss.LB}
\sum_j \|\langle \bs w_j, \bs x/2\pi\rangle\|_\xi^2 
= \E \sum_j \|\langle \bs w_j , \bs x/2\pi\rangle (\xi-\xi')\|_{\R/\Z}^2 \ge c \inf_{a_1\le |a|\le a_2} \sum_j \|a\langle \bs w_j, \bs x/2\pi\rangle\|_{\R/\Z}^2 .
\end{align}
It hence suffices to show that $\sum_j \|a\langle\bs w_j, \bs x/2\pi\rangle\|_{\R/\Z}^2\gg \log^3n$ uniformly for $|a|\in [a_1,a_2]$.
Fixing an arbitrary such $a$, since $a_1,a_2\asymp 1$ we will abuse notation and absorb $a$ into the definition of $\bs x$.
Recalling \eqref{def:wj}, 
since $\bs w_j + \bs w_{-j} = 2( 0, 0 ,\bs b_j, -(j/n)\bs a_j)$ and $\bs w_j - \bs w_{-j} = 2(\bs a_j, (j/n)\bs b_j, 0, 0)$, 
for $\bs x= (\bs x^1, \bs x^2,\bs x^3,\bs x^4)\in \R^{4\mom}$
and each $0\le j\le n$,
we have from the triangle inequality that
\begin{align*}
\|\langle \bs w_j, \bs x\rangle \|_{\R/\Z}^2 
+\|\langle \bs w_{-j}, \bs x\rangle \|_{\R/\Z}^2 
&\ge \frac12 \max\big\{  \| \langle \bs w_j + \bs w_{-j} , \bs x\rangle \|_{\R/\Z}^2\, , \, 
\| \langle \bs w_j - \bs w_{-j} , \bs x\rangle \|_{\R/\Z}^2 \big\}\\
&= 
2\max\Big\{  \| \langle  \bs b_j, \bs x^3 \rangle - (j/n)\langle  \bs a_j, \bs x^4\rangle  \|_{\R/\Z}^2 
\,,\,\|  \langle \bs a_j, \bs x^1 \rangle + (j/n) \langle \bs b_j, \bs x^2\rangle \|_{\R/\Z}^2 \Big\}.
\end{align*}
Recalling our assumption $\|\bs x\|_2\ge n^{-1/8}$, we will assume $\|\bs x^3\|_2^2 + \|\bs x^4\|_2^2 \ge \frac12 n^{-1/4}$; the complementary case that $\|\bs x^1\|_2^2 + \|\bs x^2\|_2^2 \ge \frac12 n^{-1/4}$ can be handled by the same argument. 
Fix now a vector $(\bs y, \bs y')\in \R^{2\mom}$ satisfying
\[
n^{-1/8} \le \| (\bs y, \bs y')\|_2 \le n^{K_*}
\]
and denote
\begin{equation}	\label{def:psi}
\psi(j) = \psi(j; \bs t) := \langle \bs b_j, \bs y\rangle - (j/n) \langle \bs a_j, \bs y'\rangle 
=\sum_{r=1}^\mom y_r \cos(j t_r/n) - y'_r (j/n) \sin(jt_r/n)\,.
\end{equation}
With $(\bs y, \bs y')$ playing the role of $(\bs x^3, \bs x^4)$, to establish \Cref{thm:char} our task thus reduces to establishing the following:

\begin{prop}
\label{prop:psi.LB}
Let $\bs t=(t_1,\dots, t_r) \in \R^\mom$ be $n^{\smooth}$-smooth and $\spread$-spread for some $\smooth\in (0,1)$ and $\omega(n^{-1/8\mom})\le \spread<1$. 
Then
\[
\sum_{j=0}^n \|\psi(j)\|_{\R/\Z}^2 >  \log^4n .
\]
\end{prop}

Turning to prove the proposition, we henceforth denote
\[
T:= \log^4n .
\]
In the remainder of this section we suppose towards a contradiction that 
\begin{equation}	\label{assumex}
\sum_{j=0}^n \|\psi(j)\|_{\R/\Z}^2 \le T\,.
\end{equation}

From \eqref{assumex} and Markov's inequality we have
\[
|\{ j\in [0,n]\cap \Z: \|\psi(j)\|_{\R/\Z} > 1/T\}| \le 2T^3
\]
and it follows that
there is an interval $J\subset [n]$ of length at least $n/ T^6$ such that 
\begin{equation}	\label{psij.bd}
\|\psi(j)\|_{\R/\Z} \le 1/T \qquad \forall \;j\in J.
\end{equation}
We henceforth fix such an interval $J=[n_1,n_2]$. 

Next we claim we can find $q_0\in \Z\cap [1,n^{\smooth}]$ and $s_1,\dots, s_\mom\in \R$ such that
\begin{equation}	\label{p0.claim}
q_0 t_r / 2\pi n - s_r \in \Z 
\end{equation}
and 
\begin{equation}	\label{s.claim}
\sum_{r=1}^\mom s_r^2 \le \mom n^{-2\smooth/\mom}.
\end{equation}
Indeed, considering the sequence of points $( \{ q t_1/2\pi n \} , \dots, \{qt_\mom/2\pi n\})\in [0,1]^\mom$ for $1\le q\le n^{\smooth}$, it follows from Dirichlet's principle that 
\[
\sum_{r=1}^\mom | \{ q_1 (t_r/2\pi n) \} - \{ q_2 (t_r/2\pi n)\} |^2 \le \mom n^{-2\smooth/\mom}
\]
for some $1\le q_1,q_2\le n^\smooth$.
Then we have
\[
|(q_1-q_2) t_r/2\pi n - p_r|^2 \le \mom n^{-2\smooth/\mom}
\]
for some $p_1,\dots, p_\mom\in \Z$. Now \eqref{p0.claim} and \eqref{s.claim} follow by taking $q_0=q_1-q_2$ and $s_r = (q_1-q_r) t_r/2\pi n - p_r$.

Fixing such $q_0, s_1,\dots, s_r$, we have
\begin{equation}	\label{ep0}
|e_n(q_0 t_r) -1| = |e(2\pi s_r ) -1 | \le 2\pi \mom^{1/2}n^{-\smooth/\mom} \qquad \forall\; 1\le r\le \mom.
\end{equation}

We next combine \eqref{psij.bd} and \eqref{ep0} to deduce some smoothness of the sequence $\psi(j)$ over $j\in J$, via \Cref{lem:Dbd} below.
For $g: [n]\to \C$ and positive integers $k, q$ we define the discrete differential of order $k$ and step $q$ as
\[
\Delta^k_q g: [n]\to \C\,, \qquad (\Delta^k_q g)(j) := \sum_{i=0}^k {k\choose i} (-1)^i g(j+iq).
\]
For any integer $q$ and $t\in \R$,
\begin{align*}
\sum_{i=0}^k {k\choose i} (-1)^i e_n ( (j+iq)t ) = ( 1-e_n(qt))^k e_n(jt) .
\end{align*}
Taking real parts on both sides, we obtain
\[
\sum_{i=0}^k {k\choose i} (-1)^i \cos ( (j+iq)t/n ) = \Re \big[ ( 1-e_n(qt))^k e_n(jt) \big] ,
\]
and differentiating in $t$ yields
\[
\sum_{i=0}^k {k\choose i} (-1)^i  \frac{j+iq}{n} \sin ( (j+iq)t/n ) = \Re \Big[ \partial_t \big[  ( 1-e_n(qt))^k e_n(jt) \big] \Big]\,.
\]
Combining the previous two identities over $t=t_r, r\in[m]$
we obtain the identity
\begin{equation}	\label{id:Dpsi}
(\Delta_q^k\psi)(j) 
= \Re \bigg[ \sum_{r=1}^\mom y_r (1-e_n(qt_r))^k e_n(jt_r) 
- y'_r \partial_t \big[ (1-e_n(qt_r))^k e_n(jt_r) \big] \bigg].
\end{equation}
Denoting henceforth
\begin{equation}	\label{def:fs}
f_{t,\ell}(j) := (1-e_n(\ell q_0 t))^k e_n(jt) ,
\end{equation}
substituting $q=\ell q_0$ in the above identity yields
\begin{equation}	\label{Dpsi}
(\Delta^k_{\ell q_0}\psi)(j) = \Re  \bigg[ \sum_{r=1}^\mom y_r 
\revised{f_{t_r,\ell}(j)}
+ y_r' \partial_{t_r}f_{t_r,\ell}(j) \bigg]\,
\end{equation}


\begin{lemma}
\label{lem:Dbd}
There exists $k=O_{K_*,\smooth, \mom}(1)$ such that for any $\ell \ge 1$ and any $j\in J$ such that $[j,j+k\ell q_0]\subset J$, 
\[
(\Delta_{\ell q_0}^k\psi)(j)
\ll_{K_*,\smooth, \mom} \sum_{i=0}^k \|\psi(j+i\ell q_0)\|_{\R/\Z}.
\]

\end{lemma}

\begin{proof}
Fix $k\ge 1$ to be chosen sufficiently large depending on $K_*, \smooth,\mom$. 
From \eqref{ep0}, for $\ell=1$ we have
\[
|f_{t_r, 1}(j)| \le (2\pi \mom^{1/2}n^{-\smooth/\mom})^k < n^{-k\smooth/2\mom}
\]
and
\[
|f_{t_r,1}'(j)|  \le k q_0 (2\pi \mom^{1/2}n^{-\smooth/\mom})^{k-1} + 
(2\pi \mom^{1/2}n^{-\smooth/\mom})^k < n^{-k\smooth/2\mom}
\]
and hence
\[
|(\Delta_{q_0}^k \psi)(j)| 
\le n^{-\smooth k/ 2\mom} \sum_{r=1}^\mom |y_r|+|y'_r| < m n^{K_* -\smooth k/2\mom}.
\]
Let $p(j)$ denote the closest integer to $\psi(j)$.
From the triangle inequality and \eqref{psij.bd} it follows that
\[
|(\Delta^k_{q_0} p)(j)|< m n^{K_* -\smooth k/2\mom} + \frac{2^k}T
\]
as long as $\{ j, j+q_0, \dots, j+kq_0\}\subset J$. 
Taking $k=\lf 4\mom K_*/\smooth\rf+1$, the right hand side is smaller than 1. Since the numbers $(\Delta^k_{q_0} p)(j)$ are integers, it follows that
\[
(\Delta^k_{q_0} p)(j) = 0 
\]
for all $j$ such that $\{ j, j+q_0, \dots, j+kq_0\}\subset J$. 
By repeated application of the above for $j$ running over progressions $j_0, j_0+q_0, j_0+ 2q_0,\dots$ with $j_0\in J$, we deduce that for any $j$ such that $[j,j+kq_0]\subset J=[n_1,n_2]$ there exists a polynomial $Q_j$ of degree at most $k-1$ such that
\[
p (j+iq_0) = Q_j(i)\qquad \forall\; 0\le i\le (n_2-j)/q_0.
\]
Thus we have $(\Delta^k_{\ell q_0} p)(j) = 0$ for all $\ell\ge 1$ and $j$ such that $[j,j+k\ell q_0]\subset J$. 
Hence, for such $j$ we conclude by the triangle inequality that
\[
|(\Delta^k_{\ell q_0} \psi)(j) |=
|(\Delta^k_{\ell q_0} \psi)(j) - (\Delta^k_{\ell q_0} p)(j) |
\le 2^k \sum_{i=0}^k \|\psi(j+ i\ell q_0)\|_{\R/\Z}
\]
as desired.
\end{proof}

Note that $\|\bs y\|_2+\|\bs y'\|_2\gg n^{-1/8}$. Thus either (1) there exists $i$ such that $|y_i'|\gg n^{-1/16}$ (with room to spare) or (2) $|y_i'| \le n^{-1/16}$ for all $i$ and there exists $i$ such that $|y_i| \gg_m n^{-1/8}$. In what follows we will mainly working with the first case (which is significantly harder as one needs to deal with differentials of order two). We will comment in \Cref{rmk:y1} below how to handle the second case. For the rest of the section, without loss of generality we will assume
\begin{equation}	\label{yr.bd}
|y_1'|\gg_m n^{-1/16}
\end{equation}

On the other hand, by applying \Cref{lem:Dbd} to linear combinations of shifts of $\Delta^k_{\ell q_0}\psi$ we can show the following:

\begin{lemma}
\label{lem:shifts}
For any positive integers $j,  L,L'$ and $\ell$ such that
$[j, j+ k\ell q_0 + 4(\mom-1) L +3L']\subset J$, we have
\begin{align}
&\frac{L'}n \left| y_1' \big( 1-e_n(2L't_1)\big)^2 \big( 1- e_n(\ell q_0 t_1)\big)^k
\prod_{r=2}^\mom \big( 1- e_n(L(t_1-t_r)) \big)^2\big( 1- e_n(L(t_1+t_r)) \big)^2 \right|	\notag\\
&\qquad\qquad\qquad\qquad\ll_{K_*, \smooth,\mom}
\sum_{i=1}^k \sum_{a=0}^{4(\mom-1)} \sum_{b=0}^3 \|\psi(j + i\ell q_0 + aL +bL') \|_{\R/\Z}.	\label{shifts.bd}
\end{align}
\end{lemma}

We defer the proof of \Cref{lem:shifts} for now and conclude the proof of \Cref{prop:psi.LB}.

Recall from \eqref{psij.bd} that $J=[n_1,n_2]\subset[n]$ has length $|J|\ge n/T^6$. 
Consider any $\ell\ge1$ such that $k\ell q_0\le |J|/2$.
From \Cref{lem:dilate} we can choose $\dil\asymp n/T^7 = o(|J|)$ such that
\[
\bigg\|\frac{\dil\cdot(t_r\pm t_{r'})}{2\pi n} \bigg\|_{\R/\Z} \gg_\mom \frac{\lam}{T^7}
\]
for all distinct $r,r'\in [m]$ and all choices of the signs.


Furthermore, because $t_1$ is smooth, we can choose $L'$ such that $n/T^8 \le L' =o(|J|)$ and 
\[
|1-e_n(2L't_1)| >\lam^2 =\omega(n^{-1/4\mom}).
\]
From these choices of $\ell, \dil$ and $L'$, together with \eqref{yr.bd}, we have that the left hand side in \eqref{shifts.bd} is at least 
\[
\gg_\mom n^{-1/16} \lam^4 T^{-16} (\lam/T^7)^{4(\mom-1)} |1-e_n(\ell q_0t_1)|^k.
\]


On the other hand, from \eqref{assumex} and the Cauchy--Schwarz inequality we have
\begin{equation}	\label{psi.CS}
\sum_{j=0}^n \|\psi(j)\|_{\R/\Z}\le \sqrt{nT}\,,
\end{equation}
and it follows that that we can choose $j$ so that the right hand side in \Cref{shifts.bd} is $O_{K_*,\smooth,\mom}(T^{1/2}n^{-1/2})$. 
Thus,
\begin{equation}\label{eqn:reduction}
|1-e_n(\ell q_0 t_1)| \le n^{-1/3k}
\end{equation}
and this holds for any integer $\ell\ge1$ such that $\ell kq_0\le |J|/2$. 
Applying \Cref{claim:dividing}, we conclude
\[
\|q_0t_1/2\pi n\|_{\R/\Z} = n^{-1}\log^{O(1)}n.
\]
But since we chose $q_0\le n^\smooth$ this contradicts the assumption that $t_1$ is $n^{\smooth}$-smooth.
This concludes the proof of \Cref{prop:psi.LB} and hence of \Cref{thm:char}.\qed\\


\begin{proof}[Proof of \Cref{lem:shifts}]
We begin by recording some identities. 
Recall the definition of $f_{t,\ell}(j)$ from \eqref{def:fs}.
To lighten notation we will suppress the subscript $\ell$ as it is fixed throughout the proof.
First note that
\begin{equation}	\label{partial.ft}
g_t(j):= \partial_tf_{t}(j) = \sqrt{-1} \bigg[ \frac{j}n - \frac{k\ell q_0}n \big( 1-e_n(\ell q_0t)\big)^{-1} \bigg] f_{t}(j).
\end{equation}
In particular, we have
\[
\overline{ f_{t}(j)} = f_{-t}(j) \,,\qquad 
\overline{g_{t}(j)} = -g_{-t}(j)
\]
and from \eqref{Dpsi} we can express
\begin{equation}	\label{psi12}
\frac12 (\Delta_{\ell q_0}^k\psi)(j) 
= \sum_{r=1}^\mom y_r  f_{t_r}(j) + y_rf_{-t_r}(j) 
+y'_r  g_{t_r}(j) - y'_rg_{-t_r}(j) .
\end{equation}
As in the proof of \Cref{lem:cov}
we will eliminate terms in the above sum by repeated application of the twisted second-order differencing operators defined in \eqref{def:Dt0}. 
For a positive integer $L$ and $t_0\in \R$ we have
\begin{align*}
D_{t_0} f_{t}(j)
&=\sum_{a=0}^2{2\choose a}(-1)^a e_n(-aLt_0) f_{t}(j+aL) \\
&= f_{t}(j) \sum_{a=0}^2{2\choose a}(-1)^a e_n(aL(t-t_0))\\
&=\big[1-e_n(L(t-t_0))\big]^2  f_{t}(j) .
\end{align*}
Note that the sequences $f_t(j)$ from that proof differ from the present definition by a factor $(1-e_n(\ell q_0 t))^k$. This is a key point: whereas there our aim was to lower bound $\sum_j |\psi(j)|^2$, here we have the more difficult task of lower bounding $\sum_j\|\psi(j)\|_{\R/\Z}^2$ (which we are doing by contradiction, starting from the assumption \eqref{assumex}). We are now in a similar position as in the proof of \Cref{lem:cov} thanks to \Cref{lem:Dbd} and the application of the differencing operators $\Delta_{\ell q_0}^k$, which is responsible for the extra factor $(1-e_n(\ell q_0 t))^k$. 
 

Differentiating the above expression for $D_{t_0} f_{t}(j)$ yields
\begin{align}
D_{t_0}g_{t}(j)
&=\big[1-e_n(L(t-t_0))\big]^2  \partial_tf_{t}(j) 
+ \ii\frac{L}n f_{t}(j) \sum_{a=0}^2 {2\choose a} (-1)^aa\cdot e_n(aL(t-t_0))	\notag \\
&= \big[1-e_n(L(t-t_0))\big]^2  \partial_tf_{t}(j)
- 2\ii \frac{L}n \big[1-e_n(L(t-t_0))\big] e_n(L(t-t_0)) f_{t}(j) \notag\\
&=\big[1-e_n(L(t-t_0))\big]^2 \big[ g_t(j) + \beta_L(t-t_0) f_t(j) \big]
\label{DLt0}
\end{align}
with $\beta_L(s) := -2\ii\frac Ln e_n(Ls) / [1-e_n(Ls)]$, as in \eqref{DGt}.
In particular,
\begin{equation}	\label{DL0}
D_{t_0} f_{t_0} (j)  = D_{t_0} g_{t_0}(j)=0.
\end{equation}
Now for general $t\in \R$,
two applications with $t_0$ and $-t_0$ yield
\begin{equation}	\label{D2.fab}
D_{t_0} \circ D_{-t_0} f_{t}(j)
=\big[1-e_n(L(t-t_0))\big]^2\big[1-e_n(L(t+t_0))\big]^2 f_{t}(j) 
\end{equation}
and
\begin{equation}	\label{D2.fab}
D_{t_0} \circ D_{-t_0} g_{t}(j)
=\partial_t \Big[ \big[1-e_n(L(t-t_0))\big]^2\big[1-e_n(L(t+t_0))\big]^2 f_{t}(j)\Big] .
\end{equation}
For compactness, we write
\[
\delta_L(s):= 1-e_n(Ls)
\]
for the remainder of the proof.
Applying the above identities with $t_0=t_\mom$ and $t$ running over $t_r$, $r\in [\mom-1]$, we obtain
\begin{align*}
&\frac12 \Big( D_{t_\mom} \circ D_{-t_\mom} \circ \Delta_{\ell q_0}^k\,\psi\Big)(j)   \\
&\qquad= \sum_{r=1}^{\mom-1} 
\left( y_r + y_r' \partial_{t_r} \right)
\bigg[  \delta_L(t_r-t_\mom)^2\delta_L(t_r+t_\mom)^2 f_{t_r}(j)  +
 \delta_L(-t_r-t_\mom)^2\delta_L(-t_r+t_\mom)^2 f_{-t_r}(j)  \bigg] \,.
 \end{align*}
 Iteratively applying $D_{t_r} \circ D_{-t_r}$ for $r=\mom-1,\mom-2,\dots, 2$, we get
 \begin{align*}
&\frac12\big( D_{t_2} \circ D_{-t_2}  \circ \cdots\circ D_{t_\mom} \circ D_{-t_\mom} \circ \Delta_{\ell q_0}^k\,\psi\Big)(j)  \\
&\qquad=
y_1f_{t_1}(j)  \prod_{r=2}^\mom  \delta_L(t_1-t_r)^2\delta_L(t_1+t_r)^2  +
y_1f_{-t_1}(j) \prod_{r=2}^\mom \delta_L(-t_1-t_r)^2\delta_L(-t_1+t_r)^2  \\
&\qquad+ y_1' \partial_t \bigg[  f_{t}(j)
 \prod_{r=2}^\mom  \delta_L(t-t_r)^2\delta_L(t+t_r)^2 \bigg]_{t=t_1}  +
y_1' \partial_t\bigg[ f_{-t}(j) \prod_{r=2}^\mom  \delta_L(-t-t_r)^2\delta_L(-t+t_r)^2  \bigg]_{t=t_1} \,,
 \end{align*}
 and we have passed from a sum of $4\mom$ terms (see \eqref{psi12}) to a sum of 4.
Now we will reduce from four terms to one.
Let $L'$ be a positive integer and define $D'_{t_0}$ as in \eqref{def:Dt0} with $L'$ in place of $L$.
For any univariate function $G$ we have
\begin{align*}
D'_{t_0} f_{t_0}(j) G(t_0)
& =G(t_0) D'_{t_0} f_{t_0}(j) =0\,, \\
D'_{t_0} \partial_t\big[ f_{t}(j) G(t) \big]_{t=t_0} 
&= G(t_0) D'_{t_0} g_{t_0}(j) + G'(t_0) D'_{t_0} f_{t_0}(j) = 0
\end{align*}
(using \eqref{DL0}). 
Set
\[
G(t) := \prod_{r=2}^\mom \delta_L(t-t_r)^2 \delta_L(t+t_r)^2
\]
for which we have $\overline{G(t)}=G(-t)$.
Application of $D'_{-t_1}$ to the previous expression for $\frac12( D_{t_2} \circ D_{-t_2}  \circ \cdots\circ D_{t_\mom} \circ D_{-t_\mom} \circ \Delta_{\ell q_0}^k\,\psi)(j)$ eliminates the second and fourth terms on the right hand side, leaving
 \begin{align*}
&\frac12 \Big(D'_{-t_1} \circ D_{t_2} \circ D_{-t_2}  \circ \cdots\circ D_{t_\mom} \circ D_{-t_\mom}\circ \Delta_{\ell q_0}^k\,\psi\Big)(j)  \\
&\qquad=
y_1f_{t_1}(j) \delta_{L'}(2t_1)^2 G(t_1)
+ y_1' D'_{-t_1} \partial_t \Big[  f_{t}(j)  G(t) \Big]_{t=t_1}  \\
&\qquad =
y_1f_{t_1}(j) \delta_{L'}(2t_1)^2 G(t_1)    
+ y_1'  g_{t_1}(j) \delta_{L'}(2t_1)^2 
G(t_1) + 
 y_1'  f_{t_1}(j)  \delta_{L'}(2t_1)^2
G'(t_1) \\
& \qquad= f_{t_1}(j) \bigg[ y_1 \delta_{L'}(2t_1)^2 G(t_1)  
+ y_1' 
 \sqrt{-1} \frac{j}n  \delta_{L'}(2t_1)^2 G(t_1)   \\
&\qquad\qquad\qquad\qquad - y_1' 
 \sqrt{-1} \frac{k\ell q_0}n \big( 1-e_n(\ell q_0t_1)\big)^{-1} 
 \delta_{L'}(2t_1)^2 G(t_1) 
 + y_1'    \delta_{L'}(2t_1)^2 G'(t_1) \bigg]\,,
 \end{align*}
where in the final line we substituted \eqref{partial.ft}.
Now since $f_{t_1}(j+L')= e_n(L't_1)f_{t_1}(j)$, we can eliminate all but the second term inside the brackets by multiplying both sides by $e_n(L't_1)$ and subtracting the result from the equation with $j$ replaced with $j+L'$.
We thus obtain
\begin{align*}
&\frac12 \Big(D'_{-t_1} \circ D_{t_2} \circ D_{-t_2}  \circ \cdots\circ D_{t_\mom} \circ D_{-t_\mom}\circ \Delta_{\ell q_0}^k\,\psi\Big)(j+L') \\
&\qquad\qquad - e_n(L't_1)\times\frac12 \Big(D'_{-t_1} \circ D_{t_2} \circ D_{-t_2}  \circ \cdots\circ D_{t_\mom} \circ D_{-t_\mom}\circ \Delta_{\ell q_0}^k\,\psi\Big)(j)   \\
&\qquad=
y_1'\sqrt{-1}\frac{L'}n \delta_{L'}(2t_1)^2 G(t_1) f_{t_1}(j).
\end{align*}
Recalling our definitions of $\delta_{L'}(2t_1), G(t_1),$ and  $f_{t_1}(j)$, the claimed bound now follows from taking the modulus of both sides, applying the triangle inequality to the left hand side, and applying \Cref{lem:Dbd} applied at various shifts of $\psi$. 
\end{proof}

\begin{remark}
\label{rmk:y1}
For the case that $|y_i'|\le n^{-1/16}$ and $|y_1|\gg_\mom n^{-1/8}$ in place of \eqref{yr.bd}, we can show the following simpler analogue of Lemma \ref{lem:shifts} (see also \cite[Lemma 10.5]{DNN} for a bivariate variant).
\begin{lemma}
\label{lem:shifts'} For any positive integers $j,  L,L'$ and $\ell$ such that
$[j, j+ k\ell q_0 + 4(\mom-1) L +3L']\subset J$, we have

\begin{align}
&\frac{L'}{n}\left| y_1  \big( 1- e_n(\ell q_0 t_1)\big)^k
\prod_{r=2}^\mom \big( 1- e_n(L(t_1-t_r)) \big)^2\big( 1- e_n(L(t_1+t_r)) \big)^2 \right|	\notag\\
&\qquad\qquad\qquad\qquad\ll_{K_*, \smooth,\mom}
\sum_{i=1}^k \sum_{a=0}^{4(\mom-1)} \sum_{b=0}^3 \|\psi(j + i\ell q_0 + aL +bL') \|_{\R/\Z} + O(2^kn^{-1/16}).	\label{shifts.bd'}
\end{align}
\end{lemma}
Here the additional bound $2^kn^{-1/16}$ on the RHS is caused by applying triangle inequalities basing on \eqref{id:Dpsi} (where we use $|y_i'|\ll n^{-1/16}$ for all $i$ to bound all the terms involving $\partial_t$ by $O(n^{-1/6})$ and move to the right hand side during the differential process). The proof of Lemma \ref{lem:shifts'} can be carried out exactly the same way we proved Lemma \ref{lem:shifts}, and in fact it is simpler because we don't have to take care any of the terms involving $\partial_t$ because we started with the variant of \eqref{id:Dpsi} without the $\partial_t$ term.
From Lemma \ref{lem:shifts'}, by using the assumption that $|y_1|\ge n^{-1/8}$ we can deduce \eqref{eqn:reduction}, and hence conclude  \Cref{prop:psi.LB} the same way.
\end{remark}

Before concluding this section, as our approach to prove Proposition \ref{prop:psi.LB} starts with \eqref{psij.bd}, by passing to subintervals of $J$ when needed (where we note that at least one of such subintervals still has length $\Omega(n/T^6))$, we obtain the following analogue of Theorem of Theorem \ref{thm:char} (where we recall $\phi_j(\Bx)$ from \eqref{Phi.gauss}).

\begin{theorem}[Decay of the truncated characteristic function]
\label{thm:char'}
Let $\bs t=(t_1,\dots,t_\mom)\in \R^m$ be $n^\smooth$-smooth 
and $\spread$-spread for some $\smooth\in (0,1)$ and $\omega(n^{-1/8\mom})\le \spread<1$. 
Then for any index set $I \subset [n]$ with $|I|=O(1)$, and for
any fixed $K_*<\infty$ and any $v\in \R^{4m}$ with $n^{-1/8}\le \|v\|_2\le n^{K_*}$ the following holds for sufficiently large $n$
\[
\prod_{j \notin I}| \phi_j(\Bx)| \le \exp( - \log^2n ).
\]
\end{theorem}


\section{Complex coefficients and extensions}
\label{sec:extensions}

\subsection{Theorem \ref{thm:main} when $\xi$ is complex-valued} In the case that the random coefficients are complex-valued,  our polynomial can be written as
\begin{align*}
P_n(x) &= \sum_{k=-n}^n (\xi_k+\sqrt{-1} \xi_k') (\cos(kx) + \sqrt{-1} \sin(kx)) \\
& = \xi_0 + \sqrt{-1} \xi_0' + \sum_{k=1}^n (\xi_k + \xi_{-k})\cos(kx)- (\xi_k'-\xi_{-k}') \sin(kx)\\
& + \sqrt{-1} \sum_{k=1}^n (\xi_k'+\xi_{-k}')\cos (kx) + (\xi_k-\xi_{-k}) \sin(kx)\\
\end{align*}
where $\xi_k, \xi_k'$ are iid copies $\xxi$. By limiting to only the imaginary part, the corresponding random walk of interest is 
\[
T_n(\bs t) := \sum_{j=1}^n \xi_j^{(1)} \bs u_j + \xi_j^{(2)}  \bs v_j
\]
where $\xi_j^{(1)}, \xi_j^{(2)}$ are independent sub-Gaussian of mean zero and variance one with the property that $\xi_j^{(1)}-\xi_j^{'(1)}, \xi_j^{(2)} - \xi_j^{'(2)}$ have the same distribution (here $\xi_j^{'(1)}$ and $\xi_j^{'(2)}$ are independent copies of $\xi_j^{(1)}$ and $\xi_j^{(2)}$ respectively), and where for a fixed tuple $\bs t= (t_1,\dots, t_\mom)\in \R^\mom$ and $j\in \Z$ we denote the vectors (see also \eqref{def:wj}) 
\begin{equation}
\qquad\uvec_j = \uvec_j(\bs t):=\big(\avec_j\,,\,(j/n)\bvec_j\big), \ \qquad\vvec_j = \vvec_j(\bs t):=\big(\bvec_j\,,\, -(j/n)\avec_j \big). 	\label{def:uvj}
\end{equation}
Because this random walk is only on $\R^{2m}$ with the steps $\bs u_j, \bs v_j$ compensating each other, we can establish all of our previous results under the following weakly spreading condition.

\begin{definition}
\label{def:spread:w}
For $m\ge 2$ and $\spread>0$, we say $\bs{t}=(t_1,\dots, t_\mom)\in \R^\mom$ is {\it weakly $\spread$-spread} if
\[
\Big\| \frac{t_r- t_{r'}}{2\pi n} \Big\|_{\R/\Z} \ge \frac{\spread}{n}
\qquad \forall \, 1\le r<r' \le \mom.
\]
 \end{definition}
 
 Under this condition we have the following analog of Theorem \ref{thm:char}.
 
\begin{theorem}[Decay of the characteristic function]
\label{thm:char'}
Let $\bs t=(t_1,\dots,t_\mom)\in \R^m$ be $n^\smooth$-smooth 
and weakly $\spread$-spread for some $\smooth\in (0,1)$ and $\omega(n^{-1/8\mom})\le \spread<1$. 
Then for 
any fixed $K_*<\infty$ and any $\bs x\in \R^{2m}$ with $n^{-1/8}\le \|\bs x\|_2\le n^{K_*}$,
\[
|\E e(\langle T_n(\bs t),\bs x\rangle)| \le \exp( - \log^2n )
\]
for all $n$ sufficiently large depending on $K_*, m, \kappa,$ and the sub-Gaussian constants.
\end{theorem}
We next sketch the main idea to prove this result. Fix a vector $n^{-1/8} \le \|(\bs y, \bs y')\|_2 \le n^{K_*}$, recalling \eqref{def:psi}, we further denote
\begin{equation}	\label{def:psi'}
\psi'(j): = \psi'(j; \bs t)= \langle \bs b_j, \bs y\rangle - (j/n) \langle \bs a_j, \bs y'\rangle 
=\sum_{r=1}^\mom y_r \sin(j t_r/n) + y'_r (j/n) \cos(jt_r/n)\,.
\end{equation}
The main proposition is the following analog of  Proposition \ref{prop:psi.LB}.

\begin{prop}
\label{prop:psi.LB'}
Let $\bs t=(t_1,\dots, t_r) \in \R^\mom$ be $n^{\smooth}$-smooth and assume that $\bs t$ is weakly $\spread$-spread for some $\smooth\in (0,1)$ and $\omega(n^{-1/8\mom})\le \spread<1$. 
Then
\[
\sum_{j=0}^n \|\psi(j)\|_{\R/\Z}^2 + \sum_{j=0}^n \|\psi'(j)\|_{\R/\Z}^2>  \log^4n .
\]
\end{prop}
We next sketch the proof, omitting most details. We follow the proof of Proposition \ref{prop:psi.LB} with some simplifications, that instead of focusing on $(\Delta^k_{\ell q_0}\psi)(j)$ as the real part of $\sum_{r=1}^\mom y_r f_{j,\ell}(t_r)+ y_r' \partial_{t_r}f_{t_r,\ell}(j)$ in \eqref{Dpsi} we can study the sum directly. This would allow use to shorten the differential process significantly, namely in the proof of \Cref{lem:shifts} we will only need to consider $D_{t_1}' \circ D_{t_2} \circ  \cdots\circ D_{t_\mom}$ (without negative perturbations), leading to a simpler multiplicative factor $\prod_{r=2}^\mom \big( 1- e_n(L(t_1-t_r)) \big)^2$ (without $( 1- e_n(L(t_1+t_r)))^2$), hence justifying the weakly spreadness condition. 

Finally, one can similarly prove Lemma \ref{lem:cov},  Theorem \ref{thm:box}, and Theorem \ref{thm:smallball} for the random walk $T_n(\bs t)$ above under the weakly spreadness condition on $\bs t$. Using these results, we can now conclude the proof of \Cref{prop:moments} for the complex-valued case as in Section \ref{sec:proof-moments} where we can now allow the $x_{\al_i}$ to vary entirely over $[-\pi,\pi]$.

\subsection{Other extensions}

As noted in \Cref{rmk:extensions}, with minor modifications our arguments extend \Cref{thm:main} to $P_n$ of the general form $P_n(x) = |J_n|^{-1/2} \sum_{j\in J_n} \xi_j e(jx)$ for any sequence of finite intervals $J_n\subset \Z$ with $|J_n|\to \infty$. 
By multiplying by the phase $e(-n_0x)$, which does not change the minimum modulus, where $J=[n_0, n_1]$, one sees it suffices to consider the form 
\begin{equation}	\label{Pn.gen}
P_n(x) = \frac1{\sqrt{n+1}} \sum_{j=0}^n \xi_j e(jx).
\end{equation}
Our arguments also extend to another well-studied class of trigonometric polynomials, of the form
\begin{equation}	\label{model:Trig}
P_n(x) = \frac1{\sqrt{n+a}}\bigg[ \sqrt{a} \xi_0+  \sum_{j=1}^n \xi_j \cos(jx) + \eta_j \sin (jx) \bigg]\,,
\end{equation}
where the variables $\xi_j,\eta_j$ are iid copies of a random variable $\xi$, and $a>0$ is a fixed parameter.
We note that for this model it is natural to focus only on the complex $\xi$ case as otherwise $P_n$ is likely to have roots. 

\begin{theorem}	\label{thm:gen}
\Cref{thm:main} extends to hold for $P_n$ of the forms \eqref{Pn.gen} and \eqref{model:Trig}.
\end{theorem}

For the model \eqref{model:Trig}, by combining with \Cref{thm:YZ} we obtain the following:

\begin{corollary}
\label{cor:trig}
The limit \eqref{YZ:lim} holds also for the model \eqref{model:Trig} with $\xi$ a complex variable as in \Cref{thm:main}, and $a=1/2$.
\end{corollary}

\begin{proof}
From \Cref{thm:gen} it suffices to verify that \eqref{YZ:lim} holds under $\P_{\CN_{\R}(0,1)}$. Note that under this measure, $\xi_j, j\ge0$ and $\eta_j, j\ge 1$ are iid standard complex Gaussians. 
Set $\zeta_0 = \xi_0$ and for $1\le j\le n$ set $\zeta_j:=\frac{1}{\sqrt{2}}(\xi_j+\eta_j)$, $\zeta_{-j}:= \frac1{\sqrt{2}}(\xi_j-\eta_j)$. 
From the rotational invariance of the complex Gaussian law it follows that $\zeta_j, -n\le j\le n$ are iid standard complex Gaussians. Then one verifies that with the change of variables, \eqref{model:Trig} becomes
\[
P_n(x) = \frac1{\sqrt{2n+2a}} \sum_{j=-n}^n \zeta_j e(jx).
\]
The claim now follows from the complex Gaussian case of \Cref{thm:YZ} and the choice $a=1/2$.
\end{proof}


We comment on the minor modifications of the proof of \Cref{thm:main} that are needed to obtain \Cref{thm:gen}.
The probabilistic Lemmas \ref{lem:derivatives} and \ref{lem:badarcs} follow from straightforward modifications.
\Cref{lem:neararcs} is deterministic and does not depend on the specific form of $P_n$ after conditioning on the good event.
The remainder of the argument only depends on the specific model through the 
the matrix $W$ in the definition \eqref{def:Snt} of the random walks $S_n(\bs t)$,
and the only proofs that need modification are those of \Cref{lem:cov} and \Cref{thm:char}.
For the model \eqref{model:Trig}, we may condition on $\xi_0$ and $\eta_j,j\ge 1$. 
As the trigonometric series is now real, we only need to consider a $2m$-dimensional walk of the form
\[
\sum_{j=1}^n \xi_j \bs v_j
\]
with notation as in \eqref{def:uvj}. The $n\times m$ matrix $V$ with rows $\bs v_j$ is a submatrix of $W_{[-n,n]}$ as defined in \eqref{def:wj} one checks that the argument for \Cref{lem:cov} yields the same bound on the smallest singular value of $V$. Moreover, the proof of \Cref{thm:char} began by reduction of the problem to the submatrix $V$ (see \eqref{def:psi}), so the result also holds in this case.

\appendix

\section{Separation of near-minimizers}
\label{app:neararcs}

In this appendix we prove \Cref{lem:neararcs}, restated below, along similar lines to the proof of \cite[Lemma 2.11]{YaZe}.

\begin{lemma}
On the event $\CG_{2}(\Cp/2)$ we have
\begin{enumerate}[(i)]
\item If $\CA_{\al}$ and $\CA_{\al +1}$ hold, then 
\[
Y_\al \in [\frac{\pi}{N} - \frac{\pi} {N \log^{\Cp/4} n}, \frac{\pi}{N}].
\]
\item Furthermore, $\CA_\al$ and $\CA_{\al'}$ cannot hold simultaneously as long as 
\[
2\le |\al'-\al| \le \frac{n}{\log^{3\Cp}n}.
\]
\end{enumerate}
\end{lemma}

\begin{proof}
We first show (i). Assume that $\CA_\al$ holds and  $Y_\al \in [0, \frac{\pi}{N} - \frac{\pi} {N \log^{\Cp/4} n})$. Then 
\begin{align*}
|F_\al(x_\al + \pi/N)| = |Z_\al/n + (\pi/N - Y_\al) P'(x_\a)| 
& \ge | (\pi/N - Y_\al) P'( x_\a)| - |Z_\al|/n\\
 & \gg \frac{1}{N  \log^{\Cp/4} n}  \times  \frac{n}{\log^{\Cp/2} n} -\frac{\log n}{n} \\
&\gg  \frac{\log^{\Cp/4} n}{n}- \frac{\log n}{n} \gg    \frac{\log^{\Cp/4} n}{n}.
\end{align*}
Now for $x \in I_{\al+1}$ and under $\CG_{2}(\Cp/2)$
\begin{align*}
|F_{\al+1}(x) - F_\al(x)| &\le |F_{\al+1}(x) - P(x)| +|F_\al(x) - P(x)| \\
&\ll N^{-2} \sup_{x\in [-\pi, \pi]}  |P''(x)| \\
&\ll \frac{\log^{3\Cp}n}{n^2}.
\end{align*}
So if $x\in I_{\al+1}$ then 
\begin{align*}
|F_{\al+1}(x)| 
&\ge |F_\al(x)| - |F_{\al+1}(x) - F_\al(x)| \\
&\ge |F_\al(x_\al+\pi/N)| -  |F_{\al+1}(x) - F_\al(x)|\\
&\gg \frac{\log^{\Cp/4} n}{n}, 
\end{align*}
where  $|F_\al(x)|\ge |F_\al(x_\al+\pi/N)|$ because $x_\al + \pi/N$ is closer than $x$ to the minimizer $x_\al+ Y_\al$.
The above implies that $|Z_{\al+1}| = n |F_{\al + 1}(Y_{\al+1} +x_{\al+1})| > \log n$ and hence that $\CA_{\al+1}$ does not hold.

We turn to prove (ii). For $x\in I_{\al'}$ we have 
\begin{align*}
|F_\al(x) - F_{\al'}(x)| &\le  |F_{\al}(x)-P(x)| +  |F_{\al'}(x) - P(x)| \\
&\ll (x_\al - x_{\al'})^2 \sup_{x\in [-\pi, \pi]} |P''(nx)| \\
&\le (x_\al - x_{\al'})^2 n^{2} \log^{\Cp/2}n .
\end{align*}
On the other hand, on $\CA_\al$, for all $x\in I_{\al'}$
\begin{align*}
|F_\al(x)| \ge |F_\al(x_{\al'}-\pi/N)| & \ge |F_\al(x_{\al'}-\pi/N) - F_\al(Y_a)| - |F_\al(Y_a)|\\
  & \ge |(x_{\al'}-\pi/N-Y_\al) P'(x_\al)| - |Z_\al|/n\\  
  & \gg n |x_{\al'-1}-x_\al| \log^{-\Cp/2} n - \frac{\log n}{n}\\
   & \gg n |x_{\al'-1}-x_\al| \log^{-\Cp/2} n.
   \end{align*}
Thus for all $x\in I_{\al'}$,
\begin{align*}
|F_{\al'}(x)| &\ge |F_\al(x)| - (x_\al - x_{\al'})^2 n^{2} \log^{\Cp/2}n \\ 
&\gg  |x_{\al'-1}-x_\al| n \log^{-\Cp/2} n - (x_\al - x_{\al'})^2 n^{2} \log^{\Cp/2}n \\
& \gg n |x_{\al'-1}-x_\al| ( \log^{-\Cp/2} n  -  4 |x_{\al'-1}-x_\al| n \log^{\Cp/2}n )\\
&\gg n |x_{\al'-1}-x_\al|  (  \log^{-\Cp/2} n  -  4n^{-1}\log^{3\Cp/2}n ) \\
&\gg  |x_{\al'-1}-x_\al| n  \log^{-\Cp/2} n 
 \gg \frac{ \log^{\Cp/2} n}{n},
\end{align*}
implying $|Z_{\al'}|>\log n$ and hence that $\CA_{\al'}$ does not hold.
\end{proof}



\section*{Acknowledgments} 
We thank 
Pavel Bleher, Yen Do, Oanh Nguyen, Oren Yakir and Ofer Zeitouni for helpful discussions and comments, and to Yakir and Zeitouni for showing us an early draft of their work \cite{YaZe} on the Gaussian case.
This project was initiated at the American Institute of Mathematics meeting ``Zeros of random polynomials'' in August 2019, where Bleher and Zeitouni were also participants. 
In particular, the idea used here and in \cite{YaZe} to study local linearizations emerged from those discussions.
We thank the workshop organizers and the Institute for providing a stimulating research environment. 

\bibliographystyle{amsplain}


\begin{dajauthors}

\begin{authorinfo}[nc]
  Nicholas A.\ Cook\\
  Department of Mathematics\\
 Duke University\\
 Durham, NC 27708, USA\\
  {nickcook\imageat{}math\imagedot{}duke \imagedot{}edu}
\end{authorinfo}

\begin{authorinfo}[hn]
  Hoi H. \ Nguyen\\
  Department of Mathematics\\
  The Ohio State University \\ 
  Columbus, OH 43210 USA\\
 {nguyen\imagedot{}1261\imageat{}osu\imagedot{}edu}
\end{authorinfo}

\end{dajauthors}

\end{document}